\documentclass[11pt,reqno,a4paper]{amsart}

\usepackage{graphics}

\usepackage{tikz}
\usetikzlibrary{decorations.pathreplacing,shapes,arrows,backgrounds}

\newcommand{\edgedraw}[2]{\draw (#1)--(#2);}

\newcommand{\dashedgedraw}[2]{\draw [dashed] (#1)--(#2);}

\usetikzlibrary{decorations.markings}
\usetikzlibrary{arrows,matrix}
\usetikzlibrary{snakes}
\usepgflibrary{arrows}
\tikzset{
    >=stealth',
    punkt/.style={
           rectangle,
           rounded corners,
           draw=black, very thick,
           text width=6.5em,
           minimum height=2em,
           text centered},
    pil/.style={
           ->,
           thick,
           shorten <=2pt,
           shorten >=2pt,}
}
\tikzset{every loop/.style={min distance=2mm,in=225,out=135,looseness=10}}

\usepackage{enumitem}

\usepackage{tikz}
\usetikzlibrary{decorations.markings}
\usetikzlibrary{arrows,matrix}
\usepgflibrary{arrows}

\usepackage{hyperref}
\usepackage{amssymb}
\usepackage{amsmath}
\usepackage{amsthm}
\usepackage{mathrsfs}
\usepackage{color}
\usepackage{wasysym}

\newtheorem{thm}{Theorem}[section]
\newtheorem{pro}[thm]{Proposition}
\newtheorem{lem}[thm]{Lemma}
\newtheorem{cor}[thm]{Corollary}

\theoremstyle{definition}
\newtheorem{dfn}[thm]{Definition}
\newtheorem*{acknowledgements}{Acknowledgements}

\theoremstyle{remark}
\newtheorem{rmk}[thm]{Remark}
\newtheorem{problem}{Problem}

\numberwithin{equation}{section}

\def\bbQ{\mathbb{Q}}
\def\N{\mathbb{N}}
\def\J{\mathscr{J}}
\def\D{\mathscr{D}}
\def\R{\mathscr{R}}
\def\L{\mathscr{L}}
\def\H{\mathscr{H}}
\def\cS{\mathcal{S}}
\def\cT{\mathcal{T}}
\def\cI{\mathcal{I}}
\def\cU{\mathcal{U}}
\def\cC{\mathcal{C}}
\def\cV{\mathcal{V}}
\def\cA{\mathcal{A}}
\def\cB{\mathcal{B}}
\def\eps{\varepsilon}
\def\es{\varnothing}
\def\Ra{\Rightarrow}

\def\wh#1{\widehat{#1}}
\def\Fr{Fra\"\i ss\'e\ }
\def\Frs{Fra\"\i ss\'e's\ }

\def\Age{\mathrm{Age}}

\newcommand{\cupdot}{\mathbin{\mathaccent\cdot\cup}}

\DeclareMathOperator\Aut{Aut} \DeclareMathOperator\id{id} 
\DeclareMathOperator\im{im} \DeclareMathOperator\dom{dom}
\DeclareMathOperator\fix{fix} \DeclareMathOperator\rank{rank}

\begin{document}


\title[Locally finite maximally homogeneous semigroups]{
Universal locally finite maximally homogeneous \\ semigroups and inverse semigroups 
} 


\author{IGOR DOLINKA}

\address{Department of Mathematics and Informatics, University of Novi Sad, Trg Dositeja Obradovi\'ca 4, 21101 Novi
Sad, Serbia}
\email{dockie@dmi.uns.ac.rs}

\author{ROBERT D. GRAY}

\address{School of Mathematics, University of East Anglia, Norwich NR4 7TJ, England, UK}
\email{Robert.D.Gray@uea.ac.uk}





\begin{abstract}
In 1959, P.\ Hall introduced the locally finite group $\cU$, today known as Hall's universal group. This group
is countable, universal, simple, and any two finite isomorphic subgroups are conjugate in $\cU$. It can be explicitly described as a direct limit of finite symmetric groups. It is homogeneous in the model-theoretic sense since it is the \Fr limit of the class of all finite groups. Since its introduction Hall's group, and several natural generalisations, have been widely studied. In this article we use a generalisation of \Frs theory to construct a countable, universal, locally finite semigroup $\cT$, that arises as a direct limit of finite full transformation semigroups, and has the highest possible degree of homogeneity. We prove that it is unique up to isomorphism among semigroups satisfying these properties. We prove an analogous result for inverse semigroups, constructing a maximally homogeneous universal locally finite inverse semigroup $\cI$ which is a direct limit of finite symmetric inverse semigroups (semigroups of partial bijections). 
The semigroups $\cT$ and $\cI$ are the natural counterparts of Hall's universal group for semigroups and inverse semigroups, respectively. While these semigroups are not homogeneous, they still exhibit a great deal of symmetry.
We study the structural features of these semigroups and locate several
well-known homogeneous structures within them, such as the countable generic semilattice, the countable random
bipartite graph, and Hall's group itself.
\end{abstract}


\maketitle

\section{Introduction}\label{sec_intro}

In his beautiful 1959 paper \cite{PHall} Philip Hall proved that there is a unique countable universal locally finite homogeneous group $\cU$. Here universal means that every finite group arises as a subgroup of $\cU$, and the word homogeneous is in the sense of \Fr (see \cite[Chapter~6]{Hodges}), and means that in Hall's universal group any isomorphism between finite subgroups extends to an automorphism of $\cU$. Indeed, the class of finite groups forms an amalgamation class and Hall's group is the unique \Fr limit of this class. In addition to proving the existence of this group, Hall showed that $\cU$ has many interesting properties including: 

\begin{itemize}[leftmargin=0.7cm]
\item It may be obtained concretely as a direct limit of symmetric groups by repeated applications of Cayley's Theorem in the following way. Take any group $G_0$ of order at least three. Then Cayley's Theorem can be applied to embed $G_0$ as a subgroup of the symmetric group $G_1 = \cS_{G_0}$. This process can then be repeated, embedding $G_1$ into the symmetric group $G_2 = \cS_{G_1}$ and so on. 
\item Any two isomorphic finite subgroups of $\cU$ are conjugate in $\cU$. In fact, any isomorphism between finite subgroups of $\cU$ is induced by an inner automorphism. 
\item For any $m > 1$ the set of all elements of order $m$ forms a single conjugacy class, and every element of $\cU$ can be written as a product of two elements of order $m$. This implies that $\cU$ is a simple group. 
\item It contains $2^{\aleph_0}$ distinct copies of each countable locally finite group. 
\end{itemize}
In their book \cite[Chapter~6]{KegelBook} Kegel and Wehrfritz remark that a universal locally finite group is in some sense a universe in which to do finite group theory.

%
%
%
%

Hall's group is both a direct limit of symmetric groups, and a universal locally finite simple group. It thus provides an example of central importance in the theory of infinite locally finite groups; see \cite{Hartley1995, KegelBook}. Other related work on locally finite simple groups, and direct limits of symmetric groups, may be found in \cite{Hickin1986, Kroshko1998, LeinenPuglisi2005, ThomasTD2014}. Sylow subgroups of Hall's group were investigated in \cite{Dalle1999}. 
The Grothendieck group of finitely generated projective modules over the complex group algebra of Hall's group was considered in \cite{Donkin2006}. 
More recently Hall's group has arisen in work relating to the Urysohn space \cite{Cameron2008, Doucha2016}. In \cite{Doucha2016} it is shown that there exists a universal action of Hall's locally finite group on the Urysohn space $U$ by isometries. Hall's group appears as an example in the work of Samet \cite{Samet2009} on rigid actions of amenable groups, and in the topological Galois theory developed in \cite{Caramello2016}. Interesting new results on the automorphism group of $\cU$ have been obtained in very recent work of Paolini and Shelah \cite{PaoliniShelah2017}.

The work of the present article begins with a question posed by Manfred Droste at the international conference \emph{``The 83rd Workshop on General Algebra (AAA83)''}, Novi Sad, Serbia, March 2012, who asked whether there is an analogue of Hall's universal group for semigroups. The analogue of the symmetric group in semigroup theory is the full transformation semigroup $\cT_n$ of all maps from an $n$-element set to itself under composition. One would expect, therefore, the correct semigroup-theoretic analogue of Hall's group to be a limit of finite full transformation semigroups. By Cayley's Theorem for semigroups (see \cite[Theorem~1.1.2]{Howie}) every finite semigroup is a subsemigroup of some $\cT_n$, hence any infinite limit of finite full transformation semigroups will be universal and locally finite. On the other hand such a semigroup cannot be homogeneous. In fact by \Frs Theorem, no countable universal locally finite semigroup can be homogeneous since the class of finite semigroups does not form an amalgamation class \cite[Section 9.4]{CP2}. This leads naturally to the question of how homogeneous a countable universal locally finite semigroup can be. As we shall see, there is a well-defined notion of the maximal amount of homogeneity that such a semigroup can possess. We call such semigroups maximally homogeneous. 

In more detail, if $S$ is a semigroup and $T$ is a subsemigroup of $S$, we say that $\Aut(S)$ \emph{acts homogeneously on copies of $T$} if for any subsemigroups $T_1, T_2 \leq S$, if $T_1 \cong T \cong T_2$ then every isomorphism $\phi: T_1 \rightarrow T_2$ extends to an automorphism of $S$. For fixed $S$ we can consider the class $\mathcal{C}(S)$ of isomorphism types of finite semigroups $T$ on which $\Aut(S)$ acts homogeneously. The class $\mathcal{C}(S)$ provides a measure of the level of homogeneity of $S$. As $S$ ranges over all countable universal locally finite semigroups, the classes $\mathcal{C}(S)$ form a partially ordered set under inclusion, which (as we shall see in Proposition~\ref{prop_most_hom_T} and Theorem~\ref{main_sem}) has a maximum element $\cB$. A countable universal locally finite semigroup $S$ is said to be maximally homogeneous if $\mathcal{C}(S) = \cB$.          
We shall prove that, up to isomorphism, there is a unique countable universal locally finite semigroup $\cT$, which is a limit of finite full transformation semigroups, and is maximally homogeneous. 

For inverse semigroups, the analogue of the symmetric group is the symmetric inverse semigroup $\cI_n$ of all partial bijections from an $n$-element set to itself under composition of partial maps
(see \cite{Lawson} for a general introduction to inverse semigroup theory). As for the case of semigroups, there is a well-defined notion of maximally homogeneous universal locally finite inverse semigroup.  We shall prove that there is a unique countable universal locally finite inverse semigroup $\cI$, which is a limit of symmetric inverse semigroups, and is maximally homogeneous. 

The semigroups $\cT$ and $\cI$ are the natural counterparts of Hall's universal group for semigroups and inverse semigroups, respectively. While these semigroups are not homogeneous, they exhibit a great deal of symmetry and richness in their algebraic and combinatorial structure. Since they are not homogeneous they cannot be constructed using \Frs Theorem. We instead make use of a well-known generalisation of \Frs theory called the \emph{Hrushovski construction} which, among other things, was used as the basis of the construction of 
%
%
%
some important counterexamples in model theory; see \cite{Hrushovski1993}. 
We refer the reader to \cite[Section~3]{Evans} for a description of this method. 
This 
generalisation of \Frs theory allows one to construct structures where the automorphism group acts homogeneously only on a privileged class of substructures; see \cite[Section~2.4]{Dugald}. 
%
Further generalisations are possible, including a general category-theoretic version of the \Fr construction which can be found in \cite{Droste1992} and \cite{Droste1993}. In particular, the machinery from the paper \cite{Droste1993} could also be applied to obtain the existence and uniqueness results which we give in Section~4. 

After proving the existence and uniqueness of $\cT$ and $\cI$, the rest of the article will be devoted to investigating their structure. In particular we shall see that several well-known homogeneous structures may be found in the subgroup and idempotent structure of these semigroups, including the countable generic semilattice, the countable random bipartite graph, and Hall's universal group itself.

This paper is comprised of nine sections including the introduction. In Section~2 we introduce all the necessary definitions and notation from semigroup and model theory needed for what follows. 
The notions of maximally homogeneous semigroup and inverse semigroup are discussed in Section~3 together with the connection to the notion of an amalgamation base. Section~4 is devoted to proving the existence and uniqueness of $\cT$ and $\cI$. In Section~5 we prove our main results about the structure of $\cI$ (Theorem~\ref{prop_i}). The main results about the structural properties of $\cT$ are given in Sections~6 and 7 (summarised in Theorem~\ref{prop_t}). 
%
%
%
%
In Section~8 we make some comments about the problem of determining which infinite semigroups arise as subsemigroups of $\cT$, and the corresponding question for $\cI$. 
Finally, in Section~9 we discuss the relationship between $\cT$ and full transformation limit semigroups obtained by iterating Cayley's Theorem, and the analogous question for $\cI$.


\section{Preliminaries} \label{sec_prelim} 

\subsection{Locally finite groups and semigroups} \label{subsec_hall}
%
%
%

A group $G$ is said to be \emph{locally finite} if every finite subset of $G$ generates a finite subgroup. For a comprehensive introduction to the theory of locally finite groups we refer the reader to the book \cite{KegelBook}. Here we recall just a few basic facts. 

If $G_i$ ($i \in \mathbb{N}$) is a countable collection of finite groups such that $G_i$ is a proper subgroup of $G_{i+1}$ for all $i \in \mathbb{N}$,  then the union $G = \bigcup_{i \geq 0} G_i$ of this chain
\[
G_0 \leq G_1 \leq G_2 \leq \ldots
\]
is a countably infinite locally finite group. Formally, we have a sequence of finite groups $G_0, G_1, \ldots$ and embeddings (that is, injective homomorphisms) $\sigma_i: G_i \rightarrow G_{i+1}$ 
and the union $G = \bigcup_{i \geq 0} G_i$ is the \emph{direct limit} of the \emph{direct system} $\{G_i, \alpha_i^j \}$, where $\alpha_i^j$ is defined to be $\sigma_i \sigma_{i+1} \ldots \sigma_{j-1}$ for $i<j$. For basic concepts in group theory we refer the reader to \cite{Robinson}. In particular, the definition of the direct limit of a direct system of groups may be found in  \cite[Chapter~1, pages 22-23]{Robinson}. 
%
%
%
%
%
%

Conversely, if $G$ is a countably infinite locally finite group then by enumerating the elements of $G = \{ g_0, g_1, g_2, \ldots \}$ and considering the sequence of subgroups $H_i \; (i \in \mathbb{N})$, where $H_i$ is the subgroup of $G$ generated by $\{ g_0, g_1, \ldots, g_i\}$, it is not difficult to see that there exists a countable collection of  finite subgroups $G_i \; (i \in \mathbb{N})$ of $G$, such that $G_i \leq G_{i+1}$ for all $i$, and $G = \bigcup_{i \geq 0} G_i$. For a proof of this see \cite[Lemma~1.A.9]{KegelBook}.

When we have a sequence of finite groups $G_i \; (i \in \mathbb{N})$ and embeddings $\sigma_i:G_i \rightarrow G_{i+1}$ we shall sometimes omit specific reference to the names of the mappings $\sigma_i$ and talk about a \emph{chain of embeddings} of finite groups 
\[
G_0 \rightarrow G_1 \rightarrow G_2 \rightarrow \ldots 
\]  
and speak of the direct limit of this chain which, as discussed above, may be thought of as being the union of this countable collection of finite groups, with respect to this sequence of embeddings. In the special case that each of the groups $G_i$ is isomorphic to some finite symmetric group, we say that the direct limit is an \emph{$\cS_n$-limit group}. So an $\cS_n$-limit group is a direct limit of some chain 
\[
\cS_{i_1} \to \cS_{i_2} \to \cS_{i_3} \to \dots. 
\] 
of embeddings of finite symmetric groups where $i_1 < i_2 < i_3 < \cdots$.  
The $\cS_n$-limit groups have been well studied in the theory of infinite locally finite groups \cite{Book}. In particular they give one interesting source of examples of infinite locally finite simple groups.

In this article, we shall say that a group is \emph{universal} if it embeds every finite group. By Cayley's Theorem every finite group embeds in some finite symmetric group. From this it is easily seen that any countably infinite $\cS_n$-limit group is universal. Hall's group, which was discussed in the introduction above, is a particularly nice example of a countably infinite $\cS_n$-limit group. As explained above, Hall's group $\cU$ 
may be constructed  by iterating Cayley's Theorem. Namely, let $G_0=G$ be any finite group with at least 3 elements. Then it embeds, via the right regular representation $g\mapsto\rho_g$ where $x\rho_g=xg$ for all $x\in G$, into the symmetric group $G_1=\cS_G$, which in turn embeds into $G_2=\cS_{G_1}$, and so on. In this way, we obtain a chain of embeddings of finite symmetric groups
\[
G_0 \to G_1 \to G_2 \to \dots 
\]
and Hall's group $\cU$ is the direct limit of this chain.

A semigroup $S$ is called \emph{locally finite} if every finitely generated subsemigroup of $S$ is finite. In the same way as for groups, the direct limit of a countable chain of embeddings of finite semigroups 
\[
S_0 \rightarrow S_1 \rightarrow S_2 \rightarrow \ldots 
\] 
is a countable locally finite semigroup, and every countable locally finite semigroup arises in this way. 
Given a non-empty set $X$, the collection of all mappings from $X$ to $X$ together with the operation of composition of maps forms a semigroup called the 
\emph{full transformation semigroup} $\cT_X$ on $X$. In the case that $X = [n] = \{1,2, \ldots, n \}$ we write $\cT_n$ for $\cT_X$. 
In the special case that each of the semigroups $S_i$ in the above chain of embeddings is isomorphic to some finite full transformation semigroup, we say that the direct limit is a \emph{$\cT_n$-limit semigroup},  or alternatively a \emph{full transformation limit semigroup}.
So a full transformation limit semigroup is a direct limit of some chain
\[
\cT_{i_1} \rightarrow \cT_{i_2} \rightarrow \cT_{i_3} \rightarrow \ldots 
\] 
of embeddings of finite full transformation semigroups. We shall call a semigroup \emph{universal} if it embeds every finite semigroup. By Cayley's Theorem for finite semigroups every countably infinite full transformation limit semigroup semigroup is universal.   

%
%

The \emph{symmetric inverse semigroup} $\cI_X$ on a non-empty set $X$ consists of all partial bijections from $X$ to $X$ under composition of partial maps. In the case that $X = [n] = \{1,2, \ldots, n \}$ we write $\cI_n$ for $\cI_X$. 
Inverse semigroups obtained as direct limits of chains of embeddings of symmetric inverse semigroups 
\[
\cI_{i_1} \to \cI_{i_2} \to \cI_{i_3} \to \dots
\]
where $i_1<i_2<i_3<\dots$ will be called \emph{symmetric inverse limit semigroups}, or 
\emph{$\cI_n$-limit inverse semigroups}. Any symmetric inverse limit semigroup is locally finite. 
The Vagner--Preston Theorem \cite[Theorem 5.1.7]{Howie} implies that any finite inverse semigroup embeds in some finite symmetric inverse semigroup, and from this it follows that any symmetric inverse limit semigroup is universal, in the sense that it embeds every finite inverse semigroup.

%
%
%
%
%

%
%
%

\subsection{Homogeneous structures}\label{subsec_hom}

The main objects of interest for us in this paper are structures, both algebraic and combinatorial (relational), with a high degree of symmetry. Excellent recent surveys on the subject of homogeneous structures are \cite{Evans}, and \cite{Dugald} for relational structures, and our notation and conventions follow closely the paper \cite{Evans}.

Let $\mathcal{L}$ be a countable first-order language. Given an $\mathcal{L}$-structure $M$ we shall use $M$ to denote both the structure and its domain. We use $\Aut(M)$ to denote the automorphism group of $M$. An $\mathcal{L}$-structure $M$ is called \emph{homogeneous} if whenever $A_1, A_2 \subseteq M$ are finitely generated substructures of $M$, and $f: A_1 \rightarrow A_2$ is an isomorphism, then there is an automorphism $g \in \Aut(M)$ extending $f$. A non-empty class $\cC$ of finitely generated $\mathcal{L}$-structures is called an \emph{amalgamation class} if it is closed under isomorphisms, and has three properties called the Hereditary Property (HP), Joint Embedding Property (JEP), and the Amalgamation Property (AP); see \cite[Definition~1.3]{Evans}. We shall not repeat all of these definitions here since we shall give a generalisation of these ideas in Section~\ref{sec_ALa} below. It would be good however to recall the definition of the amalgamation property here: 
\begin{itemize}[leftmargin=1.3cm]
\item[(AP)] if $A_0, A_1, A_2 \in \cC$ and $f_1: A_0 \rightarrow A_1$ and $f_2: A_0 \rightarrow A_2$ are embeddings, then there is a $B \in \cC$ and embeddings $g_1: A_1 \rightarrow B$ and $g_2: A_2 \rightarrow B$ such that $f_1 g_1 = f_2 g_2$. 
\end{itemize}
Throughout the paper, the order of composition of functions is left-to-right. Consequently, we write functions to the right of their arguments.

The age of an $\mathcal{L}$-structure $M$, denoted $\Age(M)$, is the class of all structures isomorphic to some finitely generated substructure of $M$. \Frs theorem \cite[Theorem~1.6]{Evans} (see also \cite[page 163-4]{Hodges}) says that the age of a homogeneous structure is an amalgamation class and, conversely, if $\cC$ is an amalgamation class of countable finitely generated $\mathcal{L}$-structures, with countably many isomorphism types, then there is a countable homogeneous $\mathcal{L}$-structure $M$ with age $\cC$. The structure $M$ is determined uniquely up to isomorphism by $\cC$ and is called the \emph{\Fr limit} of $\cC$, or the \emph{generic structure} of the class $\cC$. The following 
well-known examples of \Fr limits will play an important role in this article. 
%
%
%
%
\begin{itemize}[leftmargin=0.6cm]
\item The class of finite linear orders is an amalgamation class with \Fr limit $(\bbQ,\leq)$.
%
\item The class of finite semilattices is an amalgamation class. Its \Fr limit is denoted $\Omega$. 
%
%
%
\item The class of finite bipartite graphs is an amalgamation class. Its \Fr limit is the countable \emph{random bipartite graph}. This is the unique countable bipartite graph such that each part of the bipartition is infinite, and for any two finite disjoint sets $U$ and $V$ from one part, there is a vertex $w$ in the other part such that $w$ is adjacent to every vertex of $U$ and to no vertices of $V$. The random bipartite graph is not a homogeneous graph, but is 
%
%
homogeneous in the language of bipartite graphs which has an additional binary relation symbol interpreted as the bipartition. 
\item The class of finite groups is an amalgamation class. Its \Fr limit is Philip Hall's universal locally finite group $\cU$.     
\end{itemize}

\subsection{Amalgamation of semigroups and inverse semigroups} \label{subsec_amalg}

It is well known that the class of finite semigroups does not form an amalgamation class, and this applies to the class of finite inverse semigroups as well. The first example showing this for semigroups was given in the 1957 PhD thesis of Kimura; see \cite[Section 9.4]{CP2} where Kimura's example is reproduced.  
%
%
For the class of finite inverse semigroups see \cite{H75,H87} where a simple counterexample is credited to C.\ J.\ Ash. It is then an immediate consequence of \Frs Theorem that there does not exists a countable universal locally finite homogeneous semigroup,  and there does not exist a countable universal locally finite homogeneous inverse semigroup. One is then naturally led to ask the question: How homogeneous can a countable universal locally finite semigroup be? We also have the analogous question for inverse semigroups. 


Let $S$ be a countable universal locally finite semigroup and let $T$ be a finite semigroup. If $\Aut(S)$ acts homogeneously on copies of $T$ then what restriction does this put on $T$? We know that $\Aut(S)$ cannot act homogeneously on every finite $T$, but we would like $S$ to act homogeneously on as many of its subsemigroups as possible. This can be made precise via the notion of an amalgamation base, as we now explain. 

Let $\cC$ be a class of structures of a fixed signature. An \emph{amalgam in $\cC$} consists of a triple of structures $A_0, A_1, A_2 \in \cC$ and a pair of embeddings  $f_1: A_0 \rightarrow A_1$ and $f_2: A_0 \rightarrow A_2$. Often we suppress the names of the mappings and simply talk about the amalgam  
\[
A_1 \gets A_0 \to A_2. 
\]
We call $A_0$ the \emph{base} of this amalgam. 
%
%
%
%
%
%
%
%
%
If there exists a structure $B\in\cC$ and embeddings $g_1: A_1 \rightarrow B$ and $g_2: A_2 \rightarrow B$ such that $f_1 g_1 = f_2 g_2$, then we say that the 
%
%
original amalgam $A_1 \gets A_0 \to A_2$ \emph{can be embedded into some structure from $\cC$}. In this language, the class $\cC$ has the amalgamation property if any amalgam in $\cC$ can be embedded into some structure from $\cC$. 
%
%
%
%
%
%
%
%
Furthermore, a structure $A_0$ from $\cC$ is said to be an \emph{amalgamation base for $\cC$} if every amalgam in $\cC$ with base $A_0$ 
can be embedded into some structure from $\cC$.

Now, as already remarked above, finite groups have the amalgamation property, so every finite group is an amalgamation base for the class of finite groups.  However, this fails for finite semigroups, and thus the class of amalgamation bases for finite semigroups is a proper subclass of finite semigroups. 
%
Similarly, there exist finite inverse semigroups which are not amalgamation bases for the class of finite inverse semigroups. 
Let $\cA$ denote the class of all amalgamation bases for finite inverse semigroups, and let $\cB$ be the class of all amalgamation bases for finite semigroups. 

We note in passing that there is a stronger notion of amalgamation for semigroups and inverse semigroups which has also received attention in the literature, namely, that of being a \emph{strong amalgamation base}. A finite semigroup $S$ is a strong amalgamation base for the class of finite semigroups if every amalgam $A_1 \gets S \to A_2$ of finite semigroups with base $S$ can be embedded into some finite semigroup $B$ in such a way that the intersection of the images of $A_1$ and $A_2$ in $B$ is \emph{equal to} the image of $S$ in $B$. There is an analogous definition for inverse semigroups. Of course, by definition, any strong amalgamation base is an amalgamation base. Throughout this paper we shall always work with the weaker notion of amalgamation base defined above, and never with strong amalgamation bases.  

%
%
%
%


It follows from results in \cite{H78, H87} and \cite{OP} that a 
finite inverse semigroup $S$ belongs to $\cA$ if and only if $S$ is $\J$-linear, i.e.
the set of principal ideals of $S$ form a chain under inclusion. 
%
%
%
%
In particular, the symmetric inverse semigroup $\cI_X$ belongs to $\cA$ for any finite set $X$. 

%
%

\subsection{Properties of the class $\cB$} \label{subsec_propB}

A characterisation of the finite semigroups in $\cB$ is not yet known. In
this subsection we will list some examples that are known to belong to
this class. Of particular importance to the results in this paper is that
the full transformation semigroup $\cT_n$ and its opposite
$\cT_n^{opp}$ both belong to $\cB$. We recall that for a semigroup $S$, $S^{opp}$ denotes the semigroup
defined on the set $S$ with the operation $\ast$ given by $x\ast y=yx$ for all $x,y\in S$.
This is called the \emph{opposite} semigroup of $S$. 

%

\begin{lem}\label{lem_opp}
  If $S$ is an amalgamation base for the class of all finite 
  semigroups then so is $S^{opp}$. 
\end{lem}
\begin{proof}
It is easy to see that, in general, if $\phi:A \rightarrow B$ is a homomorphism between semigroups then $\phi:A^{opp} \rightarrow B^{opp}$ is also a homomorphism. Now, consider an   
   amalgam $T\gets S^{opp}\to V$ with embeddings $f$ and 
  $g$, respectively. 
%
  Then using the same mappings    
  $T^{opp}\gets S\to V^{opp}$ is also 
  an amalgam of finite semigroups. Since $S$ is assumed to belong to 
  $\cB$, there is a finite semigroup $W$ and embeddings 
  $h:T^{opp}\to W$ and $k:V^{opp}\to W$ embedding this amalgam into 
  $W$. Then $h,k$ embed the initial amalgam into 
  the finite semigroup $W^{opp}$. Hence, $S^{opp}\in\cB$. 
\end{proof} 

The results of Shoji \cite{S16} together with this lemma show that $\cT_n$ and $\cT_n^{opp}$ belong to $\cB$. It is also known that any member of $\cB$ must be $\J$-linear \cite{HP}, and that $\cB$ includes semigroup reducts of all members of $\cA$. More generally, if $S$ is a $\J$-linear semigroup and the algebra $\mathbb{C}S$ is semisimple then $S\in\cB$; see \cite{OP}.  In particular any $\J$-linear finite inverse semigroup, and so any finite group, belongs to $\cB$. Throughout the paper we will make repeated use of the fact that the semigroups mentioned in this paragraph all belong to the class $\cB$. 
On the other hand, not all $\J$-linear finite semigroups belong to $\cB$: for example, it was proved in \cite{HS} that a finite completely simple semigroup belongs to $\cB$ if and only if it is a group.

\subsection{Semigroup theory notation and definitions} \label{subsec_semi}

For general background in semigroup theory we refer the reader to \cite{Howie}. 
If $S$ is a semigroup we write $T \leq S$ to mean that $T$ is a subsemigroup of $S$, and write $T < S$ 
to mean that $T \leq S$ and $T \neq S$. Green's relations are an important tool for studying the ideal structure of semigroups. 
%
%
%
Given a semigroup $S$, we
define for $a,b\in S$:
$$
a\;\R\; b \Leftrightarrow aS^1=bS^1, \quad a\;\L\; b \Leftrightarrow S^1a=S^1b,\quad a\;\J\; b \Leftrightarrow
S^1aS^1=S^1bS^1,
$$
where $S^1$ denotes $S$ with an identity element adjoined, unless $S$ already has one.  Furthermore, we let $\H=\R\cap\L$ and $\D=\R\circ\L$, and remark that $\D$ is the join of the equivalence relations $\R$ and $\L$ because it may be shown that $\R\circ\L=\L\circ\R$. In general $\H$ is a subset of both $\R$ and $\L$, while $\R$ and $\L$ are both subsets of $\D$, which is in turn a subset of $\J$. In general the relations $\D$ and $\J$ are distinct, but for periodic semigroups they coincide; see \cite[Proposition~2.1.4]{Howie}. 
Recall that a semigroup $S$ is called periodic if for every $s \in S$ there are natural numbers $m$ and $n$ such that $s^{m+n} = s^m$. Clearly every finite semigroup, and every locally finite semigroup, is periodic.   
In particular $\J = \D$ in every locally finite semigroup. The $\R$-class of an element $a$ is denoted by $R_a$, and in a similar fashion we use the notation $L_a,J_a,H_a$ and $D_a$. An $\H$-class $H_a$ of a semigroup is a group if and only if it contains an idempotent. The group $\H$-classes are exactly the maximal subgroups of the semigroup. 

%

In situations where there is more than one semigroup under consideration we shall sometimes use the notation 
%
%
$K_a^S$ and $\mathscr{K}^S$ for $K\in\{R,L,J,H,D\}$ and $\mathscr{K}\in\{\R,\L,\J,\H,\D\}$, to specify that we are taking the relation, or equivalence class, in the semigroup $S$. 

The inclusion relation between principal ideals naturally gives rise to quasi-order relations $\leq_\mathscr{K}$
on a semigroup $S$ for $\mathscr{K}\in\{\R,\L,\J\}$: for example, we write $a\leq_\R b$ if $aS^1\subseteq bS^1$,
and similarly we define $\leq_\L$ and $\leq_\J$. Also, for two $\mathscr{K}$-classes $C$ and $C'$ we write $C\leq_\mathscr{K} C'$
%
%
if $a\leq_\mathscr{K} b$ for some  $a\in C$ and $b\in C'$.

Green's relations in in the full transformation semigroup $\cT_X$ are easy to characterise (see e.g.\ \cite[Exercise 2.6.16]{Howie}).
For $f,g\in\cT_X$ we have:
$f\;\R\; g$ if and only if $\ker(f)=\ker(g)$;
$f\;\L\; g$ if and only if $\im(f)=\im(g)$; and
$f\;\J\; g$ if and only if 
$f\;\D\; g$
if and only if 
$\rank(f)=\rank(g)$.
%
%
Here $\rank(f)=|\im(f)|$, and $\ker(f)$ is the equivalence relation on $[n]$ where $(i,j) \in \ker(f)$ if and only if $if = jf$. Consequently, $f$ and $g$ belong to the same $\H$-class if and only if both their kernels and images coincide. Let $X$ be a non-empty set and choose and fix some $\kappa\leq|X|$. Let $D_{\kappa}$ denote the $\D$-class in $\cT_X$ of all transformations of rank $\kappa$. The $\R$-classes of $D_{\kappa}$ are then indexed by the set of partitions of $X$ into $\kappa$ non-empty parts, while the $\L$-classes are indexed by the set of subsets of $X$ of cardinality $\kappa$. Given a partition $P$ with $\kappa$ non-empty parts, and a subset $A$ of $X$ of cardinality $\kappa$, we shall use $H_{P,A}$ to denote the $\H$-class given by intersecting the corresponding $\R$- and $\L$-classes of $D_{\kappa}$.   
%
%
%
%
%
%
%
%
It is well known, and easy to prove, that the $\H$-class $H_{P,A}$ is a group if and only if the set $A$ is a transversal of the partition $P$, that is, there is exactly one element from the set $A$ in each part of the partition $P$. 
We shall write $A \perp P$ to denote that $A$ is a transversal of $P$.  
%
%
%
%
If $A\perp P$ then $H_{P,A}$ is isomorphic to the symmetric group $\cS_A$. 

An element $a$ of a semigroup $S$ is \emph{regular} if there exists $b\in S$ such that $aba=a$. A semigroup $S$ is regular if all of its elements are regular. The full transformation semigroup $\cT_X$ is an example of a regular semigroup. We say that an element $a'$ is an \emph{inverse} of an element $a$ in a semigroup $S$ if both 
$aa'a=a$ and $a'aa'=a'$. It may be shown that a semigroup is regular if and only if every element has at least one inverse. An \emph{inverse semigroup} is a semigroup in which each element has exactly one inverse. Inverse semigroups are most naturally viewed as algebraic structures in the (2,1)-signature $\{\cdot,\hbox{}^{-1}\}$, where $x^{-1}$ denotes the unique inverse of $x$. Thus, we have $xx^{-1}x=x$ and $x^{-1}xx^{-1}=x^{-1}$. In addition, it may be shown that $^{-1}$ is an involution satisfying $(xy)^{-1}=y^{-1}x^{-1}$, and that the set of all idempotents of an inverse semigroup form a commutative subsemigroup; 
%
%
see \cite[Section 5]{Howie}.
%
%
An inverse subsemigroup $T$ of an inverse semigroup $S$ is a subsemigroup $T$ of $S$ which is closed under taking inverses. This is equivalent to saying that $T$ is a substructure of $S$ in the (2,1)-signature.

The Vagner--Preston Theorem \cite[Theorem 5.1.7]{Howie} shows that any inverse semigroup is a subsemigroup of some symmetric inverse semigroup. In more detail, if $S$ is an inverse semigroup, the map from $S$ to $\cI_S$ which sends each $x \in S$ to the partial bijection $\rho_x: Sx^{-1} \rightarrow Sx$, where $t \rho_x = tx$ for all $t \in Sx^{-1}$, gives an embedding of $S$ into $\cI_S$.  

 %
%
%
%
%
%
%
%
%
%

In the symmetric inverse semigroup $\cI_X$, the inverse $\alpha^{-1}$ of the element $\alpha$ is the inverse  of the mapping $\alpha$, in the usual sense. 
The $\L$-, $\D$- and $\J$-relations in $\cI_X$ are just the same as in $\cT_X$,
while we have $f\;\R\; g$ if and only if $f$ and $g$ have the same domain, which we write as $\dom(f)=\dom(g)$. Hence, for $H_f$ to be a group we must have $\dom(f)=
\im(f)$ from which it follows that $H_f$ is isomorphic to the symmetric group on $\dom(f)$. 

A $\J$-class $J$ in an arbitrary semigroup $S$ gives rise to the associated \emph{principal factor} $J^\ast = J \cup \{0\}$, with multiplication:
\[
a \cdot b = \begin{cases}
ab & \mbox{if $a$, $b$, $ab \in J$} \\
0 & \mbox{otherwise.}
\end{cases}
\] 
It is known that $J^\ast$ is either a $0$-simple semigroup or a
semigroup with zero multiplication; see \cite[Theorem
3.1.6]{Howie}. Under some additional finiteness hypotheses, such as
being periodic, a regular $0$-simple principal factor
$J^\ast$ will be completely $0$-simple
%
%
in which case the Rees Theorem \cite[Theorem 3.2.3]{Howie} states that it
will be isomorphic to a \emph{Rees matrix semigroup}
$\mathcal{M}^0[G;I,\Lambda;P]$. See \cite[Subsection 3.2]{Howie} for more details on completely $0$-simple semigroups and 
the Rees matrix semigroup construction. 
We use $B(I,G)$ to denote the Rees matrix semigroup $\mathcal{M}^0[G;I,I;P]$ where $P$ is the $I \times I$ identity matrix. This is called a \emph{Brandt semigroup} over $G$. It is a completely $0$-simple inverse semigroup, and every completely $0$-simple inverse semigroup arises in this way; see \cite[Theorem 5.1.8]{Howie}. 

%

\subsection{Graphs, posets, and semilattices} \label{subsec_gr-pos-sl}

We view graphs as structures with a single symmetric irreflexive binary relation, denoted by $\sim$. If $\Gamma$ is a graph we use $V\Gamma$ to denote its vertices and $E\Gamma$ its set of edges which are the 2-sets $\{u,v\}$ such that $u \sim v$ and $v \sim u$. 

A semigroup is called a \emph{semilattice} if it is commutative and all of its elements are idempotents. A meet-semilattice is a poset $(P,\leq)$ such that any pair of elements $x,y \in P$ has a well-defined greatest lower bound $x \wedge y$. These two definitions are equivalent: if a semigroup $S$ is a semilattice then the partial order $\leq$ given by $e \leq f$ if and only if $ef=fe=e$ is a meet semilattice, and conversely given a semilattice $(P,\leq)$ the semigroup $(P, \wedge)$ is a commutative semigroup of idempotents; see \cite[Proposition~1.3.2]{Howie}.    

\subsection{Combinatorial structures in semigroups} \label{subsec_comb-struct}

Let $D$ be a $\D$-class of a semigroup $S$. Since $\D=\R\circ\L=\L\circ\R$, we can label the $\R$-classes and the 
$\L$-classes contained in $D$ by index sets $I$ and $\Lambda$, respectively, so that each $\H$-class contained in $D$ is of the form 
$
H_{i\lambda} = R_i\cap L_\lambda.
$
To record the distribution and structure of the idempotents of $S$ within $D$, following \cite{Graham,Houghton} we define a bipartite graph associated with $D$, denoted $GH(D)$, called the \emph{Graham--Houghton graph} of $D$. The graph $GH(D)$ is defined to be the bipartite graph with vertex set the disjoint union $I \cupdot \Lambda$, where $I$ and $\Lambda$ are the two parts of the bipartition, and where 
$i \in I$ is adjacent to $\lambda \in \Lambda$ if and only if  
%
%
%
$H_{i\lambda}$ is a group. Equivalently, $i$ and $\lambda$ are adjacent in $GH(D)$ if and only if $R_i\cap L_{\lambda}$ contains an idempotent. 


Semilattices arise naturally within inverse semigroups. The set of idempotents $E(S)$ of an inverse semigroup $S$ is a commutative subsemgrioup of $S$, and therefore is a semilattice, which we shall call the \emph{semilattice of idempotents of the inverse semigroup $S$}.





\section{Universal maximally homogeneous semigroups}\label{sec_UMHS}

As discussed above, there is no countable universal locally finite
homogeneous semigroup, and there is no such inverse semigroup
either. The results in this section will describe the 
%
%
maximum
degree of homogeneity that can be possessed by a countable universal
locally finite semigroup or inverse semigroup.

Let $S$ be a semigroup and let $T$ be a subsemigroup of $S$. Recall from above that we say that $\Aut(S)$ acts homogeneously on copies of $T$ if for any subsemigroups $T_1, T_2 \leq S$, if $T_1 \cong T \cong T_2$ then every isomorphism $\phi: T_1 \rightarrow T_2$ extends to an automorphism of $S$. 

\begin{pro}\label{prop_most_hom_T}
  Let $S$ be a countable universal locally finite semigroup and let
  $T$ be a finite semigroup. If $\Aut(S)$ acts homogeneously on copies
  of $T$, then $T$ belongs to the class $\cB$ of all amalgamation bases
  for finite semigroups. 
\end{pro}
\begin{proof}
Let $f_1: T \rightarrow U_1$ and $f_2: T \rightarrow U_2$ be embeddings where $U_1$, $U_2$ are finite semigroups. Since $S$ is universal there are embeddings $g_1: U_1 \rightarrow S$ and $g_2: U_2 \rightarrow S$. Set $V_1 = U_1g_1$ and $V_2 = U_2g_2$. Since $\Aut(S)$ acts homogeneously on copies of $T$ it follows that the isomorphism $g_1^{-1}f_1^{-1}f_2g_2  :Tf_1g_1 \rightarrow Tf_2g_2$ extends to an automorphism $\alpha \in \Aut(S)$. Let $W$ be the subsemigroup of $S$ generated by $V_1 \alpha \cup V_2$. Then the mappings $g_1\alpha: U_1 \rightarrow W$ and $g_2: U_2 \rightarrow W$ complete the amalgamation diagram, and the proof.  
\end{proof}

%
%

Note that the universality hypothesis is necessary in this proposition. For example, the universal countable homogeneous semilattice $\Omega$ is a countable locally finite semigroup, while $\Aut(\Omega)$ acts homogeneously on all finite semilattices, some of which are not $\J$-linear and thus do not belong to the class $\cB$. 


%


%

The following analogue for inverse semigroups may be shown in a similar way.

\begin{pro}\label{prop_most_hom_I}
  Let $I$ be a countable universal locally finite inverse semigroup
  and let $T$ be a finite inverse semigroup. If $\Aut(I)$ acts
  homogeneously on copies of $T$, then $T$ belongs to the class $\cA$
  of all amalgamation bases for finite inverse semigroups, that is,
  $T$ is $\J$-linear.
\end{pro}

We call a universal locally finite semigroup $S$ \emph{maximally homogeneous}, or \emph{$\cB$-homogeneous},  if $\Aut(S)$ it acts homogeneously on copies of $T$ for every $T$ in $\cB$. Similarly we talk about universal locally finite inverse semigroups which are maximally homogeneous, also called $\cA$-homogeneous, meaning their automorphism group acts homogeneously on copies of $T$ for all $T$ in $\cA$.  

We would like to identify universal locally finite maximally homogeneous semigroups, and inverse semigroups, and study their properties. The most natural class of universal locally finite semigroups is given by full transformation limit semigroups, defined in Subsection~\ref{subsec_hall}. 
%
%
%
%
%
%
%
Our first main result shows that maximally homogeneous semigroups exist, and in the class of full transformation limit semigroups there is a unique example up to isomorphism. 

\begin{thm}\label{main_sem}
There is a unique maximally homogeneous full transformation limit semigroup. 
\end{thm}

This result will be proved in Section~\ref{sec_ALa}. 
%
%
We call the semigroup in this theorem \emph{the} maximally homogeneous full transformation limit semigroup, and denote it by $\cT$. 
For inverse semigroups we have the following analogous result. 
%
%
%
\begin{thm}\label{main_inv}
There is a unique maximally homogeneous symmetric inverse limit semigroup. 
\end{thm}

This result will also be proved in Section~\ref{sec_ALa}. We call the inverse semigroup in this theorem the maximally homogeneous symmetric inverse limit semigroup and denote it by $\cI$.

\begin{rmk}\label{rem_careful}
  If $S$ is a homogeneous inverse semigroup then it is obvious that
  its semilattice of idempotents $E(S)$ is a homogeneous
  semilattice. It is important to stress that $\cI$ is not a
  homogeneous inverse semigroup, so it is not immediate that $E(\cI)$
  is homogeneous. Moreover, it is not possible to prove that $E(\cI)$ is
  homogeneous simply by considering the action of $\Aut(\cI)$ on
  $E(\cI)$. Indeed, since the only semilattices in $\cA$ are chains, and since $\cI$ is universal 
  for finite inverse semigroups 
  and thus in particular embeds all finite semilattices, it
  follows from Proposition~\ref{prop_most_hom_I} that $\Aut(\cI)$ does not
  act homogeneously on all of its finite subsemilattices. Similarly, the action of $\Aut(\cT)$ on $E(\cT)$ will certainly not
give a proof that the Graham--Houghton graphs of $\cT$ are
homogeneous. Indeed, $\Aut(\cT)$ will not act homogeneously on copies
of the 2-element left zero semigroup $L_2$ inside a given $\D$-class
since $L_2$ is not in $\cB$. This is because the only finite
completely simple semigroups that belong to $\cB$ are finite groups
\cite{HS}. 
The action of $\Aut(\cT)$ induces an action on the set of idempotents which 
by restricting to a particular $\D$-class gives an action of $\Aut(\cT)$ on the Graham--Houghton graph. 
In terms of this action, the above observation says that $\Aut(\cT)$ does not 
act two-arc transitively on the Graham--Houghton graph. 
In fact, even homogeneous semigroups can have
Graham--Houghton graphs which are not homogenous. Indeed\footnote{We
  thank Thomas Quinn--Gregson of the University of York for bringing
  this example to our attention.} if $S$ is the combinatorial
completely $0$-simple semigroup represented as a Rees matrix semigroup
with structure matrix
$$
P = \begin{pmatrix} 1 & 1 & 0 & 0 \\ 0 & 0 & 1
    & 1  \end{pmatrix}
$$
then $S$ is homogeneous, but its Graham--Houghton graph is the
disjoint union of two copies of the complete bipartite graph $K_{1,2}$
which is not a homogeneous bipartite graph. 
\end{rmk}


\section{\Fr amalgamation and the proofs of the \\ existence and uniqueness of $\cT$ and $\cI$}
\label{sec_ALa}

In this section we shall make use of a generalisation of \Frs Theorem called the Hrushovski construction. 
We follow the description of this method given in \cite[Section~3]{Evans}.
We work with a class $\mathcal{K}$ of finite $\mathcal{L}$-structures and a 
distinguished class of substructures $A\sqsubseteq B$, which is expressed by saying `$A$ is a nice substructure of $B$'.
If $B \in \mathcal{K}$ then an embedding $f: A \rightarrow B$ is called a $\sqsubseteq$-embedding if $f(A) \sqsubseteq B$. 
%
%
%
%
%
%
%
%
%
%
%
%
We shall assume that $\sqsubseteq$ satisfies the following conditions:
\begin{itemize}
\item[(N1)] If $A\in\mathcal{K}$ then $A\sqsubseteq A$ (so isomorphisms are $\sqsubseteq$-embeddings).
\item[(N2)] If $A\sqsubseteq B\sqsubseteq C$ for $A,B,C\in\mathcal{K}$ then $A\sqsubseteq C$ (so, if $f:A\to B$
and $g:B\to C$ are $\sqsubseteq$-embeddings then $fg:A\to C$ is also a $\sqsubseteq$-embedding).
\end{itemize}
Note that whether an embedding $f:A\to B$ is a $\sqsubseteq$-embedding just depends on the substructure induced on $f(A)$.  
%
%
In particular, if $f:A\to B$ is a $\sqsubseteq$-embedding then so is $\alpha f:A\to B$ for any $\alpha\in\Aut(A)$. 
%
%
%
%
%
%
%
We say that $(\mathcal{K},\sqsubseteq)$ is an \emph{amalgamation class} if:
\begin{itemize}
\item $\mathcal{K}$ is closed under isomorphisms, has countably many isomorphism ty\-pes, and countably many 
embeddings between any two members of $\mathcal{K}$;
\item $\mathcal{K}$ is closed under $\sqsubseteq$-substructures;
\item $\mathcal{K}$ has the JEP for $\sqsubseteq$-embeddings: if $A_1,A_2\in\mathcal{K}$ then there exists 
$B\in\mathcal{K}$ and $\sqsubseteq$-embeddings $f_i:A_i\to B$ ($i=1,2$);
\item $\mathcal{K}$ has the AP for $\sqsubseteq$-embeddings: if $A_0,A_1,A_2\in\mathcal{K}$ and $f_1:A_0\to A_1$
and $f_2:A_0\to A_2$ are $\sqsubseteq$-embeddings then there exists $B\in\mathcal{K}$ and $\sqsubseteq$-embeddings
$g_i:A_i\to B$ ($i=1,2$) with $f_1g_1=f_2g_2$. 
\end{itemize}
%
%
%
Now it will be useful to extend the notion of a nice substructure to certain countable structures. Suppose $M$
is a countable $\mathcal{L}$-structure such that there are finite substructures $M_i$ of $M$ ($i\in\N$) with
\[
M_1\sqsubseteq M_2\sqsubseteq M_3\sqsubseteq \dots
\quad \mbox{and} \quad 
M=\bigcup_{i\in\N}M_i.
\]
For a finite $A\leq M$ we define $A\sqsubseteq M$
if 
$A\sqsubseteq M_i$ for some $i\in\N$. This does not depend on the choice of substructures $M_i$ above provided the following condition holds:
\begin{itemize}
\item[(N3)] Let $A\sqsubseteq B \in\mathcal{K}$ and $A\subseteq C\subseteq B$ with $C\in\mathcal{K}$. Then
$A\sqsubseteq C$.
\end{itemize}

\begin{thm}[Theorem 3.2 in \cite{Evans}] \label{genFr}
Suppose $(\mathcal{K},\sqsubseteq)$ is an amalgamation class of finite $\mathcal{L}$-structures and $\sqsubseteq$
satisfies (N1) and (N2). Then there is a countable $\mathcal{L}$-structure $M$ and finite substructures $M_i\in
\mathcal{K}$ ($i\in\N$) such that
\begin{enumerate}
\item $M_1\sqsubseteq M_2\sqsubseteq M_3\sqsubseteq \dots$ and $M=\bigcup_{i\in\N}M_i$;
\item every $A\in\mathcal{K}$ is isomorphic to a $\sqsubseteq$-substructure of $M$; 
\item (Extension property) if $A\sqsubseteq M$ is finite and $f:A\to B\in\mathcal{K}$ is a 
$\sqsubseteq$-embedding then there is a $\sqsubseteq$-embedding $g:B\to M$ such that $afg=a$ for all $a\in A$.
\end{enumerate}
Moreover, $M$ is determined up to isomorphism by these properties and if $A_1,A_2\sqsubseteq M$ and $h:A_1\to A_2$
is an isomorphism then $h$ extends to an automorphism of $M$ (which can be taken to preserve $\sqsubseteq$).
\end{thm}

We call this latter property \emph{$\sqsubseteq$-homogeneity} and $M$ is the \emph{generic structure} of the class $(\mathcal{K},\sqsubseteq)$. As in \Frs original theorem, this result also has a converse; see \cite[Theorem 3.3]{Evans}.

\subsection{Proof of Theorem~\ref{main_inv}} 
In this section we shall apply the above theorem to prove Theorem~\ref{main_inv}. For Theorem~\ref{main_sem} we just give a sketch of how it may be proved using the same general approach. 
%

\begin{dfn}
For $S,T\in\cA$ write $S\sqsubseteq T$ if and only if $S$ is an inverse subsemigroup of $T$.
\end{dfn}

Note that if $S\sqsubseteq T$ applies then necessarily $S$ and $T$ are both $\J$-linear. 


\begin{lem}
$(\cA,\sqsubseteq)$ is an amalgamation class and $\sqsubseteq$ satisfies (N1) and (N2).
\end{lem}

\begin{proof}
Since $\sqsubseteq$ is just the restriction of $\leq$ to the members of $\cA$ (in the (2,1)-signature),
(N1) holds trivially, and (N2) is immediate. Also, since the structures in $\cA$ are finite and the 
language $\mathcal{L}$ is finite, $\cA$ is closed under isomorphisms, contains only countably many isomorphism types and
countably many embeddings between any pair of elements of $\cA$. 
We have already seen that, vacuously, $\cA$ is 
closed under $\sqsubseteq$.

To see that $(\cA,\sqsubseteq)$ is an amalgamation class, we need to verify that it satisfies the JEP and the AP
with respect to nice embeddings. Indeed, given $A_1,A_2\in\cA$, by applying the Vagner--Preston Theorem we obtain
embeddings $f_i:A_i\to\cI_{A_i}$, $i=1,2$. Furthermore, let $g_i:\cI_{A_i}\to\cI_{A_1\cupdot A_2}$ be the
obvious natural embeddings; so $A_i$ ($i=1,2$) embeds into $B=\cI_{A_1\cupdot A_2}$ via $f_ig_i$. These are 
$\sqsubseteq$-embeddings because $\cI_{A_1\cupdot A_2}\in\cA$.

To complete the proof, we must verify that $\cA$ has the amalgamation property for $\sqsubseteq$-em\-bed\-dings. Let
$A_0,A_1,A_2\in\cA$ and $f_1:A_0\to A_1$, and $f_2:A_0\to A_2$ be $\sqsubseteq$-embeddings. Then since $A_0\in\cA$
there is a finite inverse semigroup $B$ and embeddings $g_i:A_i\to B$ ($i=1,2$) such that $f_1g_1=f_2g_2$. Now by
the Vagner--Preston Theorem there is an embedding $h:B\to\cI_B$. Then $g_ih:A_i\to\cI_B$ are $\sqsubseteq$-embeddings since
$A_i\in\cA$ by assumption and $\cI_B$ belongs to $\cA$. 
\end{proof}

Combining the above lemma with Theorem \ref{genFr} gives a countable (2,1)-algebra $\cI$ and finite inverse semigroups $P_i\in\cA$ ($i\in\N$) such that $P_1\sqsubseteq P_2\sqsubseteq \dots$ and $\cI=\bigcup_{i\in\N}P_i$. This implies $\cI$ is an inverse semigroup; furthermore, it is universal for the class of finite inverse semigroups. Moreover, it also follows from Theorem \ref{genFr} that $\cI$ is $\cA$-homogeneous, and is the unique countable $\cA$-homogeneous inverse semigroup which can be written as such a union of members $P_i$ of $\cA$. To complete the proof of Theorem~\ref{main_inv} we need to show that $\cI$ is an $\cI_n$-limit inverse semigroup by converting the chain $P_1\sqsubseteq P_2\sqsubseteq \dots$ into a chain of finite symmetric inverse semigroups.

\begin{lem}\label{lem_invLem}
Let $S$ be an $\cA$-homogeneous inverse semigroup. The following are equivalent:
\begin{itemize}
\item[(i)] $S$ is universal for finite inverse semigroups and there are finite inverse subsemigroups $S_i\in\cA$ ($i\in\N$)
such that $S_1<S_2<\dots$ and $S=\bigcup_{i\in\N}S_i$;
\item[(ii)] $S$ is an $\cI_n$-limit inverse semigroup, i.e.\ there are inverse subsemigroups $T_i$ of $S$ ($i\in\N$)
such that $T_1<T_2<\dots$ and $S=\bigcup_{i\in\N}T_i$ where each
$T_i\cong\cI_{n_i}$ for some $n_i\in\N$. 
\end{itemize}
\end{lem}

\begin{proof}
(ii)$\Ra$(i): It follows from the Vagner--Preston Theorem that $S=\bigcup_{i\in\N}T_i$ is universal for finite inverse 
semigroups, whence (i) is achieved since each $T_i\cong\cI_{n_i}$ 
%
belongs to $\cA$.

(i)$\Ra$(ii): We claim that for each $i\in\N$ there is an inverse subsemigroup $U_i$ of $S$ such that $U_i$
contains $S_i$ as an inverse subsemigroup and $U_i\cong\cI_{S_i}$. To see this, first embed $S_i$ into $\cI_{S_i}$
by the Vagner--Preston Theorem. By universality of $S$ there is an embedding $\cI_{S_i}\to S$; let $V_i$ be the image of
$\cI_{S_i}$ under this embedding. Then $V_i$ is an inverse subsemigroup of $S$ which in turn contains a 
subsemigroup $S'_i$ such that $S'_i\cong S_i$. Since $S_i\in\cA$ we can apply $\cA$-homogeneity to get an
automorphism $\alpha$ of $S$ such that $S'_i\alpha= S_i$, and then set $U_i=V_i\alpha$, completing the proof
of the claim. 

Now since each $U_i$ is finite and $S=\bigcup_{i\in\N}S_i$ there exist $1=i_0<i_1<i_2<\dots$ such that
$$
S_1=S_{i_0}\leq U_{i_0}\leq S_{i_1}\leq U_{i_1}\leq S_{i_2}\leq U_{i_2}\leq \dots
$$
Hence, we have $U_{i_0}\leq U_{i_1}\leq\dots$ and $S=\bigcup_{j\in\N_0}U_{i_j}$, where $U_{i_j}\cong 
\cI_{|S_{i_j}|}$. The proof is now completed by setting 
%
%
$T_j=U_{i_{j-1}}$ for all $j\geq 1$.
\end{proof}

This lemma tells us that $\cI$ is indeed an $\cI_n$-limit inverse semigroup, which completes the existence part
of Theorem \ref{main_inv}. Uniqueness now also follows from Theorem \ref{genFr}. Indeed, the conditions (1) and (2) of that theorem are clearly satisfied. The extension property (3) holds as a consequence of the assumption of $\cA$-homogeneity. This completes the proof of Theorem~\ref{main_inv}. 

\subsection{Proof of Theorem~\ref{main_sem}}

The proof of Theorem~\ref{main_sem} is similar to the proof of Theorem~\ref{main_inv}. For $S, T \in \cB$ we write $S \sqsubseteq T$ if and only if $S$ is a subsemigroup of $T$. Then using Cayley's Theorem for semigroups, and the definition of $\cB$, it may be seen that $(\cB,\sqsubseteq)$ is an amalgamation class and $\sqsubseteq$ satisfies (N1) and (N2). Applying Theorem~\ref{genFr} gives a countable universal locally finite semigroup $\cT$ which is $\cB$-homogeneous, that is, it is maximally homogeneous. Moreover, there is a sequence of finite semigroups $P_i \in \cB$ $(i \in \mathbb{N})$ such that $P_1 \sqsubseteq P_2 \sqsubseteq \ldots $ and $\cT = \bigcup_{i \in \mathbb{N}} P_i$. The proof of Theorem~\ref{main_sem} is then completed by the following lemma which is proved in the same way as Lemma~\ref{lem_invLem} but with $\cT_n$ in place of $\cI_n$, and Cayley's Theorem for semigroups applied instead of the Vagner--Preston Theorem. Also, for the (ii)$\Ra$(i) direction of the proof of the following result we need to appeal to the fact that for all $n$ the full transformation semigroup $\cT_n$ belongs to $\cB$. 


\begin{lem}\label{lem_semLem}
Let $S$ be a $\cB$-homogeneous semigroup. The following are equivalent:
\begin{itemize}
\item[(i)] $S$ is universal for finite semigroups and there are finite subsemigroups $S_i\in\cB$ ($i\in\N$)
such that $S_1<S_2<\dots$ and $S=\bigcup_{i\in\N}S_i$;
\item[(ii)] $S$ is a $\cT_n$-limit semigroup, i.e.\ there are subsemigroups $T_i$ of $S$ ($i\in\N$)
such that $T_1<T_2<\dots$ and $S=\bigcup_{i\in\N}T_i$ where each $T_i\cong\cT_{n_i}$ for some $n_i\in\N$.
\end{itemize}
\end{lem}



\section{Homogeneous structures within the inverse semigroup $\cI$}


The aim of this section is to prove the following result which 
gives several structural properties of the semigroup $\cI$.

\begin{thm} \label{prop_i} 
Let $\cI$ be the maximally homogeneous symmetric inverse limit semigroup. Then 
\begin{enumerate}
\item $\cI$ is locally finite and universal for finite inverse semigroups.
\item $\cI/\J$ is a chain with order type $\mathbb{Q}$.
\item Every maximal subgroup is isomorphic to Hall's group $\cU$.
\item The semilattice of idempotents $E(\cI)$ is isomorphic to the universal countable homogeneous semilattice.
\item $\J=\D$ and all principal factors are isomorphic to the Brandt semigroup $B(\N,\cU)$.
\end{enumerate}
\end{thm}
We note that $(\cI/\R,\leq_\R)$, $(\cI/\L,\leq_\L)$ and $(E(\cI),\leq)$ are all isomorphic, since $\cI$ is an inverse semigroup, so part (4) also serves as a description of the $\R$ and $\L$ orders of $\cI$. 
%
%
%
%
%
%
%
The rest of this section will be devoted to proving Theorem~\ref{prop_i}. Part (4) takes the most work, so we shall deal with it last.  Part (1) was established in the proof of Theorem~\ref{main_inv}.

\subsection{The order type of $\cI/\J$}
First note that since $\cI$ is a union of $\J$-linear inverse semigroups it follows that $\cI/\J$ is a chain. This chain is certainly countable since $\cI$ is countable. To show that it has order type $(\mathbb{Q}, \leq)$ it would now suffice to show that it is dense and without end-points. For that it suffices to observe that $\Aut(\cI)$ acts $2$-homogeneously on $\cI/\J$. In fact  $\Aut(\cI)$ acts $k$-homogeneously on $\cI/\J$ for any $k$, as we now show. 

First we record a general fact about inverse semigroups.

\begin{lem}
Let $S$ be an inverse semigroup and let $J,K$ be $\J$-classes of $S$ with $J\leq K$. Then for any $f\in E(K)$
there exists $e\in E(J)$ such that $ef=fe=e$.
\end{lem}

\begin{proof}
Pick an arbitrary idempotent $g\in J$. There exist $s,t\in S^1$ such that $sft=g$. Thus
$$
g=sft=ss^{-1}sftt^{-1}t=s(s^{-1}gt^{-1})t,
$$
so $s^{-1}gt^{-1}\,\J\, g$. Now, $e=s^{-1}gt^{-1}=(ss^{-1})f(t^{-1}t)$ is an idempotent, it belongs to $J$, and
we have $fe=ef=(ss^{-1})f(t^{-1}t)f=(ss^{-1})f^2(t^{-1}t)=(ss^{-1})f(t^{-1}t)=e$.
\end{proof}

This immediately generalises to

\begin{cor}
Let $S$ be an inverse semigroup and let $J_0<J_1<\dots<J_k$ be a chain of $\J$-classes of $S$. Then there exist 
$e_i\in J_i$, $0\leq i\leq k$, such that $\{e_0,e_1,\dots,e_k\}\leq S$ with $e_0<e_1<\dots<e_k$, i.e.\ $e_ie_j=
e_je_i=e_{\min(i,j)}$ for all $i,j$.
\end{cor}

\begin{cor}\label{cor_chains}
Given any two chains $J_0<J_1<\dots<J_k$ and $J'_0<J'_1<\dots<J'_k$ of $\J$-classes in $\cI$ there is an 
$\alpha\in\Aut(\cI)$ such that $J_i\alpha=J'_i$ for all $0\leq i\leq k$.
\end{cor}

\begin{proof}
This follows from the previous corollary  together with the condition of $\cA$-homogeneity, since the 
isomorphism between finite chains of idempotents $e_0<e_1<\dots<e_k$ and $e'_0<e'_1<\dots<e'_k$ extends to
an automorphism of $\cI$, and automorphisms map $\J$-classes onto $\J$-classes.
\end{proof}

Theorem \ref{prop_i}(2) now follows since Corollary~\ref{cor_chains} implies  that $(\cI/\J,\leq)$ is a countable dense linear order without endpoints, and thus must be isomorphic 
to $(\mathbb{Q},\leq)$.

%
 %
%


\subsection{Maximal subgroups}

By $\cA$-homogeneity, $\Aut(\cI)$ acts transitively on the set $E=E(\cI)$ as each idempotent forms a trivial subsemigroup which is $\J$-linear. It  follows that for all $e,f\in E$, $H_e\cong H_f$ i.e.\ all maximal subgroups of $\cI$ are isomorphic. 
%
%
%
Now fix $e\in E$. Since
$\cI$ is universal, it embeds every finite group. Each of these embeddings is into some group $\H$-class of $\cI$
which, since all such groups are isomorphic, implies that $H_e$ is universal. Local finiteness of $H_e$ follows
from local finiteness of $\cI$. We claim that $H_e$ is homogeneous. Indeed, if $\phi:G_1\to G_2$ is an isomorphism 
between finite subgroups of $H_e$ then $e\phi=e$. Since groups are $\J$-linear, by $\cA$-homogeneity $\phi$ extends
to $\hat\phi\in\Aut(\cI)$. Since $e\hat\phi=e$ it follows that $\hat\phi\restriction_{H_e}\in\Aut(H_e)$ and it 
extends $\phi$. Thus $H_e$ is the countable universal locally finite homogeneous group $\cU$.
This completes the proof of part (3) of Theorem~\ref{prop_i}.

\subsection{Principal factors}

Since $\cI$ is locally finite, and thus periodic, it follows that $\J = \D$ and that  every principal factor of $\cI$ is 
isomorphic to a completely 0-simple semigroup. Transitivity of $\Aut(\cI)$ on $E(\cI)$ implies that any two principal factors of $\cI$ are isomorphic. 
%
%
%
%
%
%
%
This shows every principal factor is isomorphic to a Brandt semigroup over $\cU$ (by \cite[Theorem~5.1.8]{Howie}).
Let $J$ be a $\J$-class of $\cI$. For every $n\in\N$ the finite inverse semigroup $B(\{1,\dots,n\},{1_G})$ embeds
in $\cI$, since $\cI$ is universal, and thus it embeds in $J$. From this it follows that $J$ has infinitely many $\R$- and $\L$-classes. Since
$\cI$ is countable, this proves that $J^\ast\cong B(\N,\cU)$. 
This completes the proof of part (5) of Theorem~\ref{prop_i}. 

\subsection{The semilattice of idempotents} The rest of this section will be devoted to proving part (4) of Theorem~\ref{prop_i}. 
This  requires more work than the other parts of Theorem~\ref{prop_i} due to the fact that $\Aut(I)$ does not act homogeneously on the semilattice $E(\cI)$ of idempotents; see Remark~\ref{rem_careful}. 
We shall make use of the following characterisation of the countable universal homogeneous semilattice $\Omega$.



%
%
%
%
%

\begin{thm}{\emph{(\cite[Theorem~4.2]{AB}, cf. \cite[Theorem~2.5]{Droste})}}
\label{Omega}
Let $(\Omega,\wedge)$ be a countable semilattice. Then $\Omega$ is the universal homogeneous semilattice if and
only if the following conditions hold:
\begin{itemize}
\item[(i)] no element is maximal or minimal;
\item[(ii)] any pair of elements has an upper bound;
\item[(iii)] $\Omega$ satisfies the following axiom ($\ast$) depicted in Figure~\ref{fig_PropertyStar}: for any $\alpha,\gamma,\delta,\eps\in\Omega$ such that
$\delta,\eps\leq\alpha$, $\gamma\not\leq\delta$, $\gamma\not\leq\eps$, $\alpha\not\leq\gamma$, and either
$\delta=\eps$, or $\delta\parallel\eps$ and $\gamma\wedge\eps\leq\gamma\wedge\delta$, there exists $\beta\in\Omega$
such that $\delta,\eps\leq\beta\leq\alpha$ and $\beta\wedge\gamma=\delta\wedge\gamma$ 
(in particular, $\beta\parallel\gamma$). 
\end{itemize}
\end{thm} 
Here the notation $\alpha \parallel \beta$ means that $\alpha$ and $\beta$ are incomparable, where $\alpha$ and $\beta$ are elements of a poset. 

\begin{figure}
\begin{center}
\begin{tikzpicture}
\tikzstyle{vertex}=[circle,draw=black, fill=black, inner sep = 0.3mm]
\node (Implies) at (5.5,-1.5) {\small{$\Rightarrow \exists \beta$}};
\node (a) [vertex,label={90:{\small $\alpha$}}] at (1,1.5) {};
\node (d) [vertex,label={90:{\small $\delta$}}] at (2,0) {};
\node (e) [vertex,label={90:{\small $\epsilon$}}] at (0,0) {};
\node (g) [vertex,label={90:{\small $\gamma$}}] at (4,1) {};
\node (a') [vertex,label={90:{\small $\alpha$}}] at (1+7,1.5) {};
\node (b') [vertex,label={45:{\small $\beta$}}] at (1+7,0.5) {};
\node (d') [vertex,label={90:{\small $\delta$}}] at (2+7,0) {};
\node (e') [vertex,label={90:{\small $\epsilon$}}] at (0+7,0) {};
\node (g') [vertex,label={90:{\small $\gamma$}}] at (4+7,1) {};
\node (b'g') [vertex,label={0:{\small $ \beta \land \gamma = \delta \land \gamma$}}] at (2.5+7,-1.5) {};
\node (dg) [vertex,label={0:{\small $ \delta  \land \gamma $}}] at (2.5,-1.5) {};
\node (eg) [vertex,label={-90:{\small $ \epsilon  \land \gamma $}}] at (1,-2.5) {};
\edgedraw{a}{d}
\edgedraw{a}{e}
\edgedraw{dg}{eg}
\edgedraw{b'}{d'}
\edgedraw{b'}{e'}
\edgedraw{b'}{a'}
\dashedgedraw{d}{dg}
\dashedgedraw{e}{eg}
\dashedgedraw{g}{eg}
\dashedgedraw{g}{dg}
\dashedgedraw{b'}{b'g'}
\dashedgedraw{d'}{b'g'}
\dashedgedraw{g'}{b'g'}
\end{tikzpicture}
\end{center}
\caption{Illustration of property ($\ast$) in the characterisation of the countable universal homogeneous semilattice.}\label{fig_PropertyStar}
\end{figure}
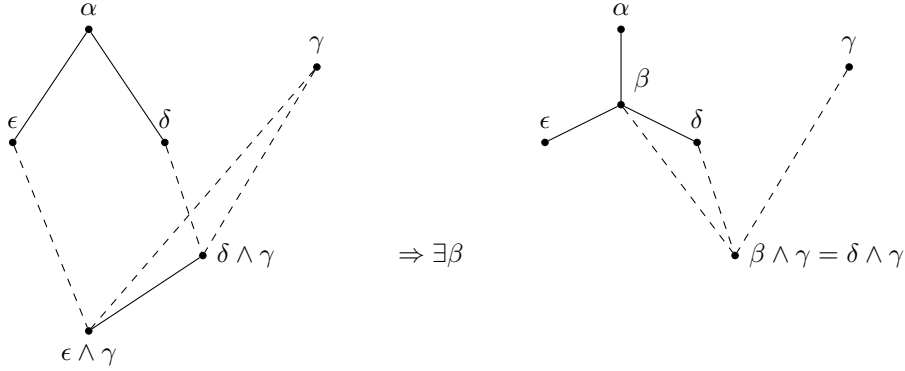

Let $\psi:S\to\cI_S$ where $x\psi=\rho_x$ be the embedding given by the Vagner--Preston Theorem. 
%
%
%
%
In the particular case when $x=e\in E(S)$ we have
$
\rho_e = \id_{Se}.
$
When $S=\cI_n$ and $\eps \in S$ is an idempotent, it follows that $\eps$ is  
the identity mapping on the subset $\dom\eps=
\im\eps$ of $\{1,\dots,n\}$, and then 
$$
\im\rho_\eps=\cI_n\eps=\{\gamma\in\cI_n:\ \im\gamma\subseteq\im\eps\}.
$$
This motivates us to define, for each $X\subseteq\{1,\dots,n\}$, the set
$$
\wh{X}=\{\gamma\in\cI_n:\ \im\gamma\subseteq X\},
$$
so that we have $\widehat{\im\eps}=\im\rho_\eps$. The operation $\ \widehat{}\ $ defines an injective map 
$$
\widehat{} \; :\mathcal{P}(\{1,\dots,n\})\to\mathcal{P}(\cI_n),
$$ 
$X\mapsto\wh{X}$. 
(Here we use the notation $\mathcal{P}(A)$ to denote the \emph{power set} of all subsets of a set $A$.) 
It is easy to see that for all $X,Y\subseteq\{1,\dots,n\}$ we have 
$
\widehat{X\cap Y} = \wh{X}\cap\wh{Y},
$
from which it follows that $\ \widehat{}\ $ is a semilattice embedding from $E(S)$ into $E(\cI_S)$.  

In the special case that $S=\cI_n$, the semilattice $E(\cI_n)$ is isomorphic to 
$(\mathcal{P}(\{1,\dots,n\}),\cap)$ via the map $\eps\mapsto\im\eps$ and, similarly, 
$E(\cI_{\cI_n})$ is isomorphic to $(\mathcal{P}(\cI_n),\cap)$. We conclude that $\ \widehat{}\ $ gives rise to a semilattice 
embedding of $E(\cI_n)$ into $E(\cI_{\cI_n})$.

\begin{pro}\label{sets}
Suppose we have $A,C,D,E\subseteq\{1,\dots,n\}$ such that the conditions of the left-hand side of axiom ($\ast$)
are satisfied, that is $D\cup E\subseteq A$, $C\not\subseteq D$, $C\not\subseteq E$, $A\not\subseteq C$, and 
either $D=E$, or $D\parallel E$ and $C\cap E\subseteq E\cap D$. Then setting $B=\wh{D}\cup\wh{E}\subseteq \cI_n$
we have $\wh{D}\subseteq B\subseteq\wh{A}$, $\wh{E}\subseteq B$ and $B\cap\wh{C}=\wh{D}\cap\wh{C}$.
\end{pro}

\begin{proof}
Both $\wh{D}\subseteq B$ and $\wh{E}\subseteq B$ are obvious from the definition of $B$. Since $D\subseteq A$ we have $\wh{D}\subseteq\wh{A}$,
and since $E\subseteq A$ we have $\wh{E}\subseteq\wh{A}$, thus $B=\wh{D}\cup\wh{E}\subseteq\wh{A}$. To complete
the proof, observe that
$$
B\cap\wh{C}=(\wh{D}\cup\wh{E})\cap\wh{C}=(\wh{D}\cap\wh{C})\cup(\wh{E}\cap\wh{C})=\wh{D\cap C}\cup\wh{E\cap C}.
$$
If $D=E$ then this equals $\wh{D\cap C}=\wh{D}\cap\wh{C}$. Otherwise, $D\parallel E$ and $E\cap C\subseteq D\cap C$
which implies $\wh{E\cap C}\subseteq \wh{D\cap C}$ and so again we obtain 
$
B\cap\wh{C}=\wh{D\cap C}=\wh{D}\cap\wh{C}.
$
\end{proof}

\begin{proof}[Proof of Theorem \ref{prop_i} (4)]
Since $\cI$ is an $\cI_n$-limit inverse semigroup it follows that for any $e,f\in E(\cI)$ there exists 
$g\in E(\cI)$ such that $g\geq e$ and $g\geq f$. Since $\Aut(\cI)$ acts transitively on $E(\cI)$ it follows that there
are no maximal or minimal idempotents. So it just remains to verify property ($\ast$).
Suppose $\alpha,\gamma,\delta,\eps\in E(\cI)$ satisfy the conditions of the left-hand side of axiom ($\ast$).
Since $\cI$ is an $\cI_n$-limit inverse semigroup there exists $S\leq\cI$ with $S\cong \cI_n$ for some $n$
and $\{\alpha,\gamma,\delta,\eps\}\subseteq S$. Let $\psi:S\to\cI_S$ be the  Vagner--Preston embedding. It follows from
Proposition \ref{sets} that there is an element $\beta\in E(\cI_S)$ such that $\rho_\alpha,\rho_\gamma,\rho_\delta,
\rho_\eps,\beta$ (where $\rho_x=x\psi$) satisfy the right-hand side of ($\ast$). Since $S\cong\cI_n$ and $\cI_S$
are both $\J$-linear we can apply the extension property of $\cI$ (Theorem~\ref{genFr} part (3)) to obtain $S\leq V\leq\cI$ where $V\cong \cI_S$
and $\beta'\in E(V)$ such that $\alpha,\gamma,\delta,\eps,\beta'$ satisfy the right-hand side of axiom ($\ast$).
Now by Theorem \ref{Omega} it follows that $E(\cI)$ is the countable universal homogeneous semilattice.
\end{proof}



\section{Homogeneous structures within the semigroup $\cT$}

In this section we shall prove some results about the structure of the
universal maximally homogeneous full transformation limit semigroup $\cT$.

%
%

\begin{thm} \label{prop_t} 
Let $\cT$ be the maximally homogeneous full transformation limit semigroup. Then 
\begin{enumerate}
\item $\cT$ is locally finite and universal for finite semigroups.
\item $\cT/\J$ is a chain with order type $\mathbb{Q}$.
\item Every maximal subgroup is isomorphic to Hall's group $\cU$.
\item $\cT$ is regular and idempotent generated, $\J=\D$ and all principal factors are isomorphic to each other.
\item The Graham--Houghton graph of every $\D$-class
  of $\cT$ is isomorphic to the countable random bipartite
  graph.
\end{enumerate}
\end{thm}
%
In the process of proving part (5) of Theorem~\ref{prop_t}, another interesting structural property of $\cT$ that we shall establish is that it is isomorphic to its opposite; see Theorem~\ref{thm_flip}. 
Parts (1), (3) and (4) are the most straightforward to prove, so we begin with them. 

\begin{proof}[Proofs of parts (1), (3) and (4) of Theorem~\ref{prop_t}] 

Part (1) is an immediate consequence of the definition of $\cT$ and
Cayley's Theorem for semigroups.    

(3) The argument is very similar to Theorem~\ref{prop_i}(3). Every finite
group belongs to $\cB$. In particular the trivial group belongs to
$\cB$ which implies $\Aut(\cT)$ acts transitively on $E(\cT)$. It
follows that all the maximal subgroups of $\cT$ are isomorphic to each
other. Fix and idempotent $e$ in $\cT$ and consider the group
$H_e$. The group $H_e$ is universal and locally finite because $\cT$
has both of these properties. Since every finite group belongs to
$\cB$, we can then see that the group $H_e$ is homogeneous as a
consequence of $\cB$-homogeneity of $\cT$.

(4) Since a union of regular semigroups is regular, and $\cT_n$ is
regular, it follows that $\cT$ is regular. 
To see that $\cT$ is idempotent generated first recall that, by \cite{HowieIGPaper}, 
for every $n \in \mathbb{N}$ we have $\langle
E(\cT_n) \rangle = (\cT_n \setminus \cS_n)\cup\{\mathrm{id}_{[n]}\}$. Let $n > 1$ and consider
the embedding $f: \cT_{n-1} \rightarrow \cT_n$ where for each $\alpha \in
\cT_{n-1}$ the mapping $\alpha f$ is given by
\[
i (\alpha f) = 
\begin{cases}
i \alpha & \mbox{if $1 \leq i \leq n-1$} \\
n & \mbox{if $i=n$.}
\end{cases}
\]   
Since the full transformation semigroup is in $\cB$ the map $f$ is a 
 $\sqsubseteq$-embedding and hence by the extension property, for
 every embedding $\cT_{n-1} \sqsubseteq \cT$ there is a
 $\sqsubseteq$-embedding $g:\cT_n \rightarrow \cT$ such that $a fg = a$
 for all $a \in \cT_{n-1}$. Since $\cT$ is a $\cT_n$-limit semigroup, for every element $s \in \cT$ 
 there is a subsemigroup $U \leq \cT$ with $s \in U$ and $U \cong \cT_k$. Setting $n=k+1$ this 
 gives an embedding $\cT_{n-1} \sqsubseteq \cT$ such that the image of $\cT_{n-1}$ contains the 
 element $s$. But now since \[f(\cT_{n-1}) \subseteq \cT_n \setminus \cS_n \subseteq \langle E(\cT_n) \rangle\]
 it follows that $s \in \langle E(\cT_ng) \rangle$. Since $s$ was arbitrary this proves that $\cT$ is 
 an idempotent generated semigroup.     
Since $\cT$ is locally finite, it is periodic, and hence $\J=\D$. 
%
%
Finally, 
transitivity of $\Aut(\cT)$ on $E(\cT)$ implies that any two principal factors of $\cT$ are isomorphic.  
\end{proof}

\subsection{Proof of Theorem~\ref{prop_t}(2)}

Clearly, a union of $\J$-linear semigroups is $\J$-linear, and thus  
in particular every $\cT_n$-limit semigroup $S$ has the
property that $S / \J$ is a chain. 
%
%
To show that $\cT / \J$ has order type $\mathbb{Q}$, we first observe that as a consequence of \cite[Theorem~1]{HP} we have the following lemma. 

\begin{lem}
If $\phi:S\to T$ is an embedding of regular semigroups and $x,y\in S$ are such that $J_x^S>J_y^S$ then
$J_{x\phi}^T>J_{y\phi}^T$.
\end{lem}

Now we prove a result analogous to the one holding for $\cI$ regarding the action of the 
automorphism group on chains of $\J$-classes, but with a different argument.

\begin{lem}\label{lem_chains}  
Let $r<m\leq n\in\N$. In $\cT_n$, for any $f\in E(D_m)$ there is an $e\in E(D_r)$ such that $ef=fe=e$.
\end{lem}

\begin{proof}
Assume that
$$
f=\left(\begin{array}{cccc}
A_1 & A_2 & \cdots & A_m\\
a_1 & a_2 & \cdots & a_m
\end{array}\right),
$$
where $a_i\in A_i$. This notation means that $f$ is a transformation from $\cT_n$ with image $\{a_1, a_2, \ldots, a_m\}$ 
such that, for each $i$, the preimage of $a_i$ is the set $A_i$. Since $f$ is an idempotent it follows that $a_i \in A_i$ for all $i$. 
Now if we set set
$$
e=\left(\begin{array}{ccccc}
A_1 & A_2 & \cdots & A_{r-1} & A_r\cup\dots\cup A_m\\
a_1 & a_2 & \cdots & a_{r-1} & a_r
\end{array}\right),
$$
then $e$ is an idempotent since $a_r \in A_r\cup\dots\cup A_m$, we have $e \in D_r$ since $|\im(e)| = r$, and it may then easily be verified that 
$ef=fe=e$.
\end{proof}

\begin{cor}\label{cor_chain}
Given a chain $J_{i_1}<J_{i_2}<\dots <J_{i_m}$ of $\J$-classes in $\cT_n$, 
for all $1 \leq j \leq m$
there exist $e_{i_j}\in J_{i_j}$ such
that $\{e_{i_1},\dots,e_{i_m}\}$ is a subsemigroup isomorphic to an $m$-element chain.
\end{cor}

\begin{lem}\label{J-chains-in-T}
For any two chains $J_0<J_1<\dots<J_k$ and $J'_0<J'_1<\dots<J'_k$ of $\J$-classes in $\cT$ there exists
$\alpha\in\Aut(\cT)$ such that $J_i\alpha=J'_i$ for all $0\leq i\leq k$.
\end{lem}
\begin{proof}
In $\cT$, choose representatives $a_i\in J_i$ and $b_i\in J'_i$, $0\leq i\leq k$. Since $\cT$ is a $\cT_n$-limit
semigroup, it is a union of its subsemigroups $\cT=\bigcup_{j\geq 0}T_j$ such that $T_j\cong\cT_{i_j}$. So there
is an $m\in\N$ such that $a_0,a_1,\dots,a_k,b_0,b_1,\dots,b_k\in T_m\cong\cT_{i_m}$. 
Now, by Corollary~\ref{cor_chain}, within $T_m$ we have idempotents $e_i\J a_i$ and $f_i\J b_i$ such that 
$e_0<e_1<\dots<e_k$ and $f_0<f_1<\dots<f_k$. Now $\phi:\{e_0,\dots,e_k\}\to\{f_0,\dots,f_k\}$, $e_i\phi=f_i$, is
an isomorphism between two members of $\cB$, since by \cite{OP} any chain semilattice belongs to $\cB$, and thus it
extends to $\hat\phi\in\Aut(\cT)$ with $J_i\hat\phi=J'_i$ for $0\leq i\leq k$.
\end{proof}
Since $\cT$ is countably infinite,  Theorem~\ref{prop_t}(2) is now an immediate consequence of Lemma~\ref{J-chains-in-T}. 


\section{The Graham--Houghton graph of $\J$-classes in $\cT$}

In this section we prove part (5) of Theorem~\ref{prop_t}. Since the principal factors of $\cT$ are all isomorphic to each other, there is (up to isomorphism) only one Graham--Houghton graph 
$\Gamma=\Gamma(\cT)$ to investigate. Our aim in this section is to show that $\Gamma$ is isomorphic to 
the countable random
bipartite graph. This bipartite graph was defined in Subsection~\ref{subsec_hom} above. 
We shall find it useful to make use of the following alternative characterisation of the countable random bipartite graph. 

Let $\Gamma$ be a bipartite graph with bipartition $V\Gamma = X \cupdot Y$. We say $\Gamma$ satisfies property $(\diamond)$ if 
\begin{enumerate}
\item $|X| = |Y| = \aleph_0$, 
\item for every pair of non-empty finite subsets $A$ and $B$ of $X$ with $A \cap B = \varnothing$ there is a vertex  
$y \in Y$ such that $y \sim a$ for all $a \in A$, and $y \not\sim b$ for all $b \in B$, and 
\item for every pair of non-empty finite subsets $C$ and $D$ of $Y$ with $C \cap D = \varnothing$ there is a vertex 
$x \in X$ such that $x \sim c$ for all $c \in C$, and $y \not\sim d$ for all $d \in D$.
\end{enumerate}
It is easy to see that this is equivalent to the defining property of the countable random bipartite graph given in Subsection~\ref{subsec_hom}. Thus, any countable bipartite graph satisfying property $(\diamond)$ is isomorphic to the random bipartite graph. 

Throughout, let $D$ be a fixed $\D$-class of $\cT$. We use $\Lambda$ to denote the set of all $\L$-classes of
$D$, while $I$ will stand for the set of $\R$-classes of $D$. Recall  from Subsection \ref{subsec_comb-struct} that $\Gamma=GH(D)$ has $V\Gamma=I\cupdot\Lambda$ and
edges $(i,\lambda) \in E\Gamma$ if and only if $H_{i\lambda}$ is a group.

Our aim is in fact to prove that $\Gamma$ is has property $(\diamond)$, namely:
\begin{itemize}
\item[(a)] $|I| = |\Lambda| = \aleph_0$;
\item[(b)] for any non-empty subsets $\Omega,\Sigma\subseteq\Lambda$ with $\Omega\cap\Sigma=\es$ there exists $i\in I$
such that all $H_{i\omega}$ ($\omega\in\Omega$) are groups and none of the $H_{i\sigma}$ ($\sigma\in\Sigma$)
are groups;
\item[(c)] the dual of (b) with $I$ and $\Lambda$ interchanged.
\end{itemize}

We will show that property (a) follows from the construction of $\cT$. 
We will prove (b) directly, while we will provide an indirect proof of (c) by showing that $\cT\cong\cT^{opp}$ and then appeal to (b).
Proving (b) relates to the following combinatorial question.
Recall that if $P$ is a partition of $[m] =\{1,2,\ldots,m\}$, and $A$ is a susbet of $[m]$,  we write $P \perp A$ to mean that $A$ is a transversal of $P$. 
Let $A_1, A_2, \ldots, A_k, B_1, B_2, \ldots, B_l$ be a family of distinct $t$-element subsets of $[m]$. Then one can ask under what conditions on these sets can we guarantee that there is a partition $P$ of $[m]$ into $t$ non-empty parts, such that $P \perp A_i$ for all $1 \leq i \leq k$ and $P \not\perp B_j$ for all $1 \leq j \leq l$? For example, if the sets $A_i,B_j$ are all pairwise disjoint then it is easy to find such a partition $P$. More generally, we can give a sufficient condition for such a set $P$ to exist, given by measuring the extent to which the sets $A_i,B_j$ overlap with each other. 
\begin{lem}[Flower Lemma]\label{lem_flower}
Let $A_1, A_2, \ldots, A_k, B_1, B_2, \ldots, B_l$ be a family of distinct $t$-element subsets of $[m]$ with $k,l \geq 1$ and  $t \geq 2$. 

For each $1 \leq i \leq k$ set 
\[
A_i' = A_i \setminus (A_1 \cup \ldots \cup A_{i-1} \cup A_{i+1} \cup 
\ldots \cup
A_k \cup B_1 \cup \ldots \cup B_l),
\]
and for $1 \leq j \leq l$ set 
\[
B_j' = B_j \setminus (A_1 \cup \ldots \cup A_k \cup B_1 \cup \ldots 
\cup B_{j-1} \cup B_{j+1} \cup \ldots 
\cup B_l).
\]
Let $Y = A_1 \cup \ldots \cup A_k \cup B_1 \cup \ldots \cup B_l$, let $Y' = A_1' \cup \ldots \cup A_k' \cup B_1' \cup \ldots \cup B_l'$ and set $M = Y \setminus Y'$. If $|M| < t$ then there exists a partition $P$ of $[m]$ into $t$ non-empty parts, such that $P \perp A_i$ for all $1 \leq i \leq k$ and $P \not\perp B_j$ for all $1 \leq j \leq l$.
\end{lem}
\begin{proof}
We call $Y$ the flower, $M$ the flower head and the sets $A_i'$, $B_j'$ the petals; see Figure~\ref{fig_flower}. 
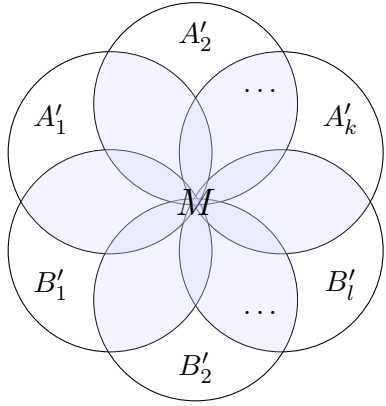
\begin{figure}

\def\firstcircle{(390:1.3cm) circle (1.34cm)}
\def\secondcircle{(90:1.3cm) circle (1.34cm)}
\def\thirdcircle{(150:1.3cm) circle (1.34cm)}
\def\fourthcircle{(210:1.3cm) circle (1.34cm)}
\def\fifthcircle{(270:1.3cm) circle (1.34cm)}
\def\sixthcircle{(330:1.3cm) circle (1.34cm)}

\begin{tikzpicture}

        \draw \firstcircle node {};
        \draw \secondcircle node {};
        \draw \thirdcircle node {};
        \draw \fourthcircle node {};
        \draw \fifthcircle node {};
        \draw \sixthcircle node {};

    \begin{scope}[fill opacity=0.3]
      \clip \firstcircle;
      \fill[blue!15] \secondcircle;
    \end{scope}

    \begin{scope}[fill opacity=0.3]
      \clip \secondcircle;
      \fill[blue!15] \thirdcircle;
    \end{scope}

    \begin{scope}[fill opacity=0.3]
      \clip \thirdcircle;
      \fill[blue!15] \fourthcircle;
    \end{scope}

    \begin{scope}[fill opacity=0.3]
      \clip \fourthcircle;
      \fill[blue!15] \fifthcircle;
    \end{scope}

    \begin{scope}[fill opacity=0.3]
      \clip \fifthcircle;
      \fill[blue!15] \sixthcircle;
    \end{scope}
    
    \begin{scope}[fill opacity=0.3]
      \clip \sixthcircle;
      \fill[blue!15] \firstcircle;
    \end{scope}

\node (A1) at (150:2.2cm) {$A_1'$};
\node (A2) at (90:2.2cm) {$A_2'$}; 
\node (dots) at (60:1.7cm) {$\ldots$};  
\node (Ak) at (390:2.2cm) {$A_k'$};   

\node (B1) at (210:2.2cm) {$B_1'$};
\node (B2) at (270:2.2cm) {$B_2'$}; 
\node (dots) at (-60:1.7cm) {$\ldots$};  
\node (Bk) at (330:2.2cm) {$B_l'$}; 

\node (M) at (0:0cm) {\Large $M$};

\end{tikzpicture}

\caption{Illustration of Lemma~\ref{lem_flower}. The set $Y$ decomposes into the disjoint union of the shaded part $M$, the head, and the non-shaded parts $A_i'$ and $B_i'$, the petals. }\label{fig_flower}
\end{figure} 
The definitions in the statement  decompose the flower into a disjoint union of its head and petals as follows:
\[
Y = M \cupdot A_1' \cupdot \ldots \cupdot A_k' \cupdot B_1' \cupdot \ldots \cupdot B_l'. 
\] 
Since $|A_i|= |B_j| = t$ and $|M| < t$ it follows that each of the petals 
$A_i'$ and $B_j'$ is non-empty. Note that in general $Y$ could be a proper subset of $[m]$. We construct a partition $P$ with the desired properties in the following way. Begin with sets $P_1, P_2, \ldots, P_t$ all empty. We will describe an algorithm for adding all the elements from $[m]$ to these sets, in such a way as to create non-empty sets $P_1, P_2, \ldots, P_t$ defining the required partition $P$ of $[m]$.

\begin{itemize}[leftmargin=0.8cm]
\item[(i)] Place the elements from $M$ in distinct sets, say $P_1, \ldots, P_{|M|}$. This is possible since $t > |M|$. 
\item[(ii)] For each $1 \leq i \leq k$ the elements from $A_i \setminus A_i'$ have already been distributed among the $P_i$, each one in a distinct set. We now add the remaining elements from $A_i'$ to the sets $P_i$ in such a way that $A_i$ is a transversal of the sets $P_i$. This is possible since there are $t$ sets, and $|A_i|=t$.    
\item[(iii)] For each $1 \leq j \leq l$, add all of the elements from $B_j'$ to the classes in such a way that $B_j$ is not a transveral of the family of sets $P_i$. If $B_j'=B_j$, since $|B_j| = t \geq 1$ this may done by assigning all of the elements to the set $P_1$. Otherwise, choose $b \in B_j \setminus B_j'$ and put all of the elements from the non-empty set $B_j'$ into the unique set $P_j$ which contains $b$. 
\item[(iv)] Put all of the elements from $[m] \setminus Y$ into the class $P_1$. 
\end{itemize}
It is now easy to see that the partition $P$ with parts $P_1, P_2, \ldots, P_t$ has the property that $A_i \perp P$ for all $1 \leq i \leq k$ while $P \not\perp B_j$ for all $1 \leq j \leq l$.
\end{proof}

The main vehicle for proving (b) above (and thus Theorem~\ref{prop_t}(5)) is the following result about finite full transformation semigroups.
 
\begin{pro}\label{gh-a}
Let $n\in\N$ be such that $n\geq 3$ and let $1<r<n$. Then there is an $m\in\N$ and an embedding $\phi:\cT_n\to
\cT_m$ such that for any collection of elements $a_1,\dots,a_k,b_1,\dots,b_l\in D_r\subseteq\cT_n$, with $k\geq 1$
and $l\geq 1$, all coming from distinct $\L$-classes, there exists an element $c\in\cT_m$ such that in $\cT_m$ we have $c\D a_1\phi$
and $c\D b_1\phi$, and $R_c\cap L_{a_i\phi}$ is a group for all $1\leq i\leq k$ while $R_c\cap 
L_{b_j\phi}$ is not a group for all $1\leq j\leq l$.
\end{pro}

\begin{proof}
Set $X=\{1,\dots,n\}$; also, let $Y$ be a finite non-empty set
with $Y\cap X=\es$. Define $\iota:\cT_X\to\cT_{X\cup Y}$ such that $\alpha\iota$ is given by 
$$
z(\alpha\iota) = \left\{\begin{array}{ll}
z\alpha & \text{if }z\in X,\\
z & \text{if }z\in Y.
\end{array}\right.
$$
Note that $\iota$ is an injective semigroup homomorphism, and that 
for any $\alpha\in\cT_X$ we have $\rank(\alpha\iota)=\rank\alpha+|Y|$.

Fix $e\in E(D_r)$ in $\cT_X$. Let $R_e$ be the $\R$-class of $e$ in $\cT_X$. Set $\eps=e\iota\in\cT_{X\cup Y}$.
The semigroup $\cT_X$ acts on the set $R_\eps\cup\{0\}$ 
where the action
$(R_\eps\cup\{0\})\times\cT_X\to R_\eps\cup\{0\}$ is given by
$$
\gamma\cdot\alpha = \left\{\begin{array}{ll}
\gamma(\alpha\iota) & \text{if }\gamma(\alpha\iota)\in R_\eps,\\
0 & \text{otherwise.}
\end{array}\right.
$$
Here, $\gamma(\alpha\iota)$ is simply a product of two elements of $\cT_{X\cup Y}$. Using the fact that $\iota$ is a homomorphism, it follows easily that this is  a right action of $\cT_X$ on $R_\eps\cup\{0\}$.

Set $Z=R_\eps\cup\{0\}$. The semigroup $\cT_X$ acts on $X$ in the obvious way and it also acts on $Z$ as above; thus $\cT_X$
acts on $X\cupdot Z$. This action gives rise to a homomorphism $\psi:\cT_X\to\cT_{X\cup Z}$ where for $\alpha\in\cT_X$ the mapping 
$\alpha\psi$
is given by 
$$
u(\alpha\psi) = \left\{\begin{array}{ll}
u\alpha & \text{if }u\in X,\\
u\cdot\alpha & \text{if }u\in Z=R_\eps\cup\{0\}.
\end{array}\right.
$$
The homomorphism $\psi$ is injective since if $\alpha,\beta\in\cT_X$ with $\alpha\neq\beta$ then there is an 
$x\in X$ with $x\alpha\neq x\beta$. Consequently, $x(\alpha\psi)\neq x(\beta\psi)$, so $\alpha\psi\neq\beta\psi$.

\begin{lem}\label{lem_really}
Let $\alpha\in D_r\subseteq \cT_X$. Then
$$
\im(\alpha\psi) = \im\alpha\cup H_{\alpha\iota}^{\cT_{X\cup Y}}\cup\{0\}.
$$
\end{lem}

\begin{proof}
The restriction of $\alpha \psi$ to $X$ gives $X \alpha = \im \alpha$. Now consider the set $Z (\alpha \psi)$, that is, the image of the set $Z = R_\eps\cup\{0\}$ under $\alpha \psi$. For any $q \in R_\epsilon$ if $q (\alpha \psi) = q \cdot \alpha \neq 0$ then $q \cdot \alpha = q (\alpha \iota) \in H_{\alpha\iota}^{\cT_{X\cup Y}}$. Therefore $Z (\alpha \psi) \subseteq  H_{\alpha\iota}^{\cT_{X\cup Y}}\cup\{0\}$. 

Conversely, since $\cT_{X \cup Y}$ is regular, there is an idempotent $f$ in the $\R$-class of $\alpha \iota$ in $\cT$. Now $f (\alpha \iota) = \alpha \iota$ and hence by Green's Lemma \cite[Lemma 2.2.1]{Howie}, in $\cT_{X \cup Y}$, right multiplication by $\alpha \iota$ defines a bijection $L_f \rightarrow L_{\alpha \iota}$ which maps the $\H$-class $R_\varepsilon \cap L_f$ bijectively onto the set $H_{\alpha\iota}^{\cT_{X\cup Y}}$. Thus $H_{\alpha\iota}^{\cT_{X\cup Y}}\cup\{0\} \subseteq Z (\alpha \psi)$. This completes the proof. 
\end{proof}

\noindent \emph{Claim.} 
Let $\alpha,\beta\in D_r\subseteq \cT_X$. If $(\alpha,\beta)\not\in\L$ then $H_{\alpha\iota}^{\cT_{X\cup Y}}\cap
H_{\beta\iota}^{\cT_{X\cup Y}}=\es$.
\begin{proof}[Proof of claim]
Since $\cT_{X \cup Y}$ is a regular semigroup and $\cT_X$ is a regular subsemigroup of $\cT_{X \cup Y}$, if  $(\alpha\iota,\beta\iota)\in\L$ it would follow that $(\alpha,\beta)\in\L$ (see \cite[Proposition 2.4.2]{Howie}), a contradiction.  
\end{proof}

Note that 
$$
|H_{\alpha\iota}^{\cT_{X\cup Y}}|=(\rank\eps)!=(r+|Y|)!
$$
In particular, this can be made arbitrarily large by varying $Y$.

Now choose $Y$ so that $(r+|Y|)!>n+1$ and let $a_1,\dots,a_k,b_1,\dots,b_l\in D_r\subseteq\cT_X=\cT_n$ be as in
the statement of Proposition \ref{gh-a}. Further, let $A_i=\im a_i$ and $B_i=\im b_i$ for all $i$. Set
$C_i=H_{a_i\iota}$ and $D_i=H_{b_i\iota}$. By the claim immediately above 
for any $i,j$ we have $C_i\cap C_j=\es$,
$C_i\cap D_j=\es$ and $D_i\cap D_j=\es$ whenever the sets are distinct. By Lemma~\ref{lem_really}, we have  
$$
\im(a_i\psi)=A_i\cup C_i\cup\{0\},
\quad 
\im(b_i\psi)=B_i\cup D_i\cup\{0\},
$$
and since all $a_i$ and $b_j$ are elements of $\cT_n$ it follows that  
$$
|A_1\cup\dots\cup A_k\cup B_1\cup\dots\cup B_l\cup\{0\}|\leq n+1.
$$

We are now in a position to apply Lemma~\ref{lem_flower}.  
Set 
$\mathbb{A}_i = \im(a_i\psi)=A_i\cup C_i\cup\{0\}$ for $1 \leq i \leq k$ and 
$\mathbb{B}_i = \im(a_i\psi)=B_i\cup D_i\cup\{0\}$ for $1 \leq i \leq l$. 
For $i \neq j$ we have 
\[
C_i\cap C_j=\es, \; C_i\cap D_j=\es \ \mbox{and}  \ D_i\cap D_j=\es, 
\]
and for all $i$ we have 
\[
C_i \cap (A_1\cup\dots\cup A_k\cup B_1\cup\dots\cup B_l\cup\{0\}) = \es, 
\]
and
\[
D_i \cap (A_1\cup\dots\cup A_k\cup B_1\cup\dots\cup B_l\cup\{0\}) = \es.  
\]
In the terminology of Lemma~\ref{lem_flower} the sets $C_i$ and $D_j$ are subsets of the petals. Specifically, for all $i$ and $j$ we have
$
C_i \subseteq \mathbb{A}_i'
$
and
$
D_j \subseteq \mathbb{B}_j'. 
$
Also in the terminology and notation of Lemma~\ref{lem_flower} we have the head of the flower $M$ satisfies 
\[
M \subseteq A_1\cup\dots\cup A_k\cup B_1\cup\dots\cup B_l\cup\{0\}
\]
and since $A_i = \im(a_i)$, $B_j = \im(b_j)$ with $a_i, b_j \in D_r$ we have 
$
M \subseteq \{1, 2, \ldots, n \} \cup \{ 0 \}
$
and so $|M| \leq n+1$. We chose $Y$ so that $(r+|Y|)!>n+1$. It follows that 
\[
|\mathbb{A}_i| = |\mathbb{B}_i| = r + (r + |Y|)! + 1 > (r + 1) + (n + 1) \geq (r+1) + |M| > |M|.   
\]
So we have a family of subsets $\mathbb{A}_i$, $\mathbb{B}_j$ of $X \cup Z$ all of size 
\[
|\mathbb{A}_i| = |\mathbb{B}_i| = r + (r + |Y|)! + 1. 
\]
such that $|\mathbb{A}_i| = |\mathbb{B}_j| > |M|$ for all $i$ and $j$. Thus the conditions of Lemma~\ref{lem_flower} are satisfied. Applying Lemma~\ref{lem_flower} we conclude that there exists a partition  $P$ of $X \cup Z$ into $r + (r + |Y|)! + 1$ non-empty parts, such that $P \perp A_i$ for all $1 \leq i \leq k$ and $P \not\perp B_j$ for all $1 \leq j \leq l$. Let $c \in \cT_{X \cup Z}$ be a mapping with $\ker(c) = P$. Then we have $X = \{1,2, \ldots, n\}$ and we have constructed an injective homomorphism $\psi: \cT_X \rightarrow \cT_{X \cup Z}$, and found an element $c \in \cT_{X \cup Z}$ such that $c\D a_1\psi$ and $c\D b_1\psi$ and $R_c\cap L_{a_i\psi}$ is a group for all $1\leq i\leq k$ and $R_c\cap 
L_{b_j\psi}$ is not a group for all $1\leq j\leq l$. This completes the proof of Proposition~\ref{gh-a}.   
\end{proof}  

\subsection{Proof of Theorem~\ref{prop_t}(5)}
We need to show that the Graham--Houghton graph $\Gamma$ has property $(\diamond)$. For this we need to show that the conditions 
(a), (b) and (c)  all hold.  Recall that $D$ is a fixed $\D$-class of $\cT$, $\Lambda$ is the set of all $\L$-classes of
$D$, and $I$ is the set of $\R$-classes of $D$. The bipartite graph $\Gamma$ has $V\Gamma=I\cupdot\Lambda$ and
edges $(i,\lambda)\in E\Gamma$ if and only if $H_{i\lambda}$ is a group. 

Property (a), that $|I| = |\Lambda| = \aleph_0$, follows easily from the construction of $\cT$. Indeed, since, as already observed, all the principal factors are isomorphic to each other, and $\cT$ is regular and universal for finite semigroups, every finite right zero semigroup embeds in $D$ and since in any such semigroup all the elements must belong to distinct $\H$-classes it follows that $D$ has infinitely many $\L$-classes. Similarly, since arbitrary finite left zero semigroups embed, $D$ also has infinitely many $\R$-classes. There are countably many in each case since $\cT$ itself is countable. 

To prove property (b) we shall apply Proposition~\ref{gh-a}. Let $\Omega$ and $\Sigma$ be disjoint non-empty subsets of $\Lambda$. Let $\alpha_1,\dots,\alpha_k$ be a transversal of the $\L$-classes $\Omega$, and $\beta_1,\dots,\beta_l$ be a transversal of the $\L$-classes $\Sigma$. Since $\cT$ is a $\cT_n$-limit semigroup, there exists a subsemigroup $S\leq\cT$ such that $S\cong\cT_n$ for some $n$ and with $O = \{ \alpha_1, \ldots, \alpha_k, \beta_1, \ldots, \beta_l\} \subseteq S$. Since the elements from the set $O$ all come from distinct $\L$-classes if $\cT$ and $S \leq \cT$, it follows that the elements from $O$ all belong to distinct $\L$-classes of the semigroup $S$. Applying Proposition~\ref{gh-a} to $\alpha_1, \ldots, \alpha_k, \beta_1, \ldots, \beta_l$ in $\cT_n$ we conclude that there is an $m\in\N$ and an embedding $\psi:S\to\cT_m$ such that there exists $c\in\cT_m$ such that in $S$ we have $c\D \alpha_1\psi$,
$c\D \beta_1\psi$, $R_c\cap L_{\alpha_i\psi}$ is a group for all $1\leq i\leq k$, and $R_c\cap 
L_{\beta_j\psi}$ is not a group for all $1\leq j\leq l$. Now by the extension property there exists an embedding $\theta:\cT_m\to\cT$ such that $\psi\theta=\id_S$. Let $i_0 \in I$ be the index of the $\R$-class of $c\theta$ in $\cT$. We claim that (b) is satisfied by taking $i_0 \in I$. For this, we use the following general fact about periodic semigroups. 

\begin{lem}
Let $S$ and $T$ be periodic semigroups and let $f:S\to T$ be an embedding. Then for
all $x\in S$, $H_x^S$ is a group if and only if $H_{xf}^T$ is a group. 
\end{lem}

\begin{proof}
Since $S$ is periodic, $H_x^S$ is a group if and only if $x^k=x$ for some $k>1$. However, since $f$ is an isomorphism between $S$ and $Sf$, the latter condition holds
if and only if $(xf)^k=xf$, which holds if and only if $H_{xf}^T$ is a group.
\end{proof}

Applying this lemma it follows that in $\cT$ we have that $R_{c\theta} \cap L_{\alpha_i}$ is a group for all $1\leq i\leq k$ and $R_{c\theta} \cap L_{\beta_j}$ is not a group for all $1\leq j\leq l$. This completes the proof of (b). 

To prove (c) it will suffice to show $\cT\cong\cT^{opp}$. Recall that from Lemma~\ref{lem_opp} in Subsection~\ref{subsec_propB} it follows that for any $n \geq 1$ both $\cT_n$ and $\cT_n^{opp}$ belong to the class $\cC$.   

\begin{thm}\label{thm_flip}
Let $\cT$ be the maximally homogeneous full transformation limit semigroup. Then
$\cT\cong\cT^{opp}$.
\end{thm}
\begin{proof}
Upon writing $\cT=\bigcup_{j\in\N}T_j$ as a union of subsemigroups such that $T_j\cong\cT_{i_j}$ for some $i_j$
we easily conclude that $\cT^{opp}=\bigcup_{j\in\N}T_j^{opp}$ where $T_j^{opp}\cong\cT_{i_j}^{opp}$. Hence,
$\cT^{opp}$ is a direct limit of semigroups from $\cB$. The semigroup $\cT^{opp}$ is also clearly countable and is universal for finite semigroups by the left Cayley Theorem for semigroups. It now follows from Lemma~\ref{lem_semLem} that if $\cT^{opp}$ is $\cB$-homogeneous then $\cT^{opp} \cong \cT$. We now prove that $\cT^{opp}$ is $\cB$-homogeneous by verifying that it has the $\cB$-extension property, which is equivalent to $\cB$-homogeneity since $\cT^{opp}$ is countable.  

Let $S$ be a finite subsemigroup of $\cT^{opp}$ with $S\in\cB$. Then $S$ is a subsemigroup of $T_j^{opp}$ for
some $j$, implying that $S^{opp}\in\cB$ is a subsemigroup of $T_j\cong\cT_{i_j}$. 
Observe that the domains of $S$ and $S^{opp}$ are equal. 
Consider an embedding $S\to U
\in\cB$; the same mapping is an embedding $S^{opp}\to U^{opp}\in\cB$. Now apply the extension property for $\cT$
to conclude that there is an embedding $U \rightarrow \cT$, which is the identity mapping when restricted to the set $S^{opp}$. But now this same map is an embedding $U^{opp} \rightarrow \cT^{opp}$ which is the identity mapping when restricted to the set $S$. This completes the proof. 
\end{proof}

Condition (c) now follows from (b) by applying Theorem~\ref{thm_flip}. Indeed, choose and fix an isomorphism $\theta:\cT \rightarrow \cT^{opp}$. Fix a $\D$-class of $\cT$. Since $\Aut(\cT^{opp})$ acts transitively on the set of idempotents in $\cT^{opp}$, and $\cT^{opp}$ is regular, we may choose $\theta$ so that the image of the set $D \subseteq \cT$ under $\theta$ is equal to the same set $D \subseteq \cT^{opp}$. Restricted to the set $D$ the map $\theta$ maps $E(D)$ bijectively to $E(D)$ in such a way that for any pair of idempotents $e,f$ we have $e \R f$ in $\cT$ if and only if $e \theta \L f \theta$ in $\cT^{opp}$, and dually $e \L f$ in $\cT$ if and only if $e \theta \R f \theta$. It follows that the mapping $\theta$ induces an automorphism of the Graham--Houghton graph of $D$ which swaps the two parts of the bipartition. The existence of this automorphism, together with the fact that we have already established property (b) for $\Gamma$, implies that $\Gamma$ also satisfies property (c). 
This completes the proof that $\Gamma$ is the
countable universal homogeneous bipartite graph and thus concludes the proof of Theorem~\ref{prop_t}.

Note that one consequence of the argument given in the previous paragraph is the following combinatorial result about finite transformation semigroups, which is the natural left-right dual to Proposition~\ref{gh-a}.  

\begin{cor}
Let $n,r\in\N$ with $n\geq 3$ and $1<r<n$. Let $a_1,\dots,a_k,b_1,\dots,b_l\in D_r$ (where $k,l \geq 1$) be representatives of distinct
$\R$-classes of $D_r\subseteq\cT_n$. Then there exists $m\in\N$ and an embedding $\psi:\cT_n\to\cT_m$ such that 
there is a $c\in\cT_m$ with the property that $R_{a_i\psi}\cap L_c$ is a group for all $i$ and $R_{b_j\psi}\cap
L_c$ is not a group for all $j$. 
\end{cor}
It is not obvious how we would prove this corollary directly and combinatorially (i.e.\ how to find $m$, $\psi$ and $c$) except by going via the argument above which uses the isomorphism $\cT \cong \cT^{opp}$. It relates to the following general combinatorial problem: Let $P_1, P_2, \ldots, P_k, Q_1, Q_2, \ldots, Q_l$ be a family of distinct partitions of $[m]=\{1,2,\ldots,m\}$ each with exactly $t$ non-empty parts. Under what conditions on these partitions can one guarantee that there is a $t$-element subset $A$ of $[m]$ such that $A \perp P_i$ for all $1 \leq i \leq k$ and $A \not\perp Q_j$ for all $1 \leq j \leq l$?

\subsection{The principal factors of $\cT$} We have seen that all the principal factors of $\cT$ are isomorphic to each other, and that their Graham--Houghton graphs are isomorphic to the countable random bipartite graph. Let $J$ be a fixed $\J$-class of $\cT$. We do not currently have any characterisation of the principal factor $J^*$. Since $\cT$ is locally finite, and $\cT$ is regular, it follows that $J^*$ is a completely $0$-simple semigroup. 
If $\mathcal{CS}$ is the class of all finite completely $0$-simple semigroups then $(\mathcal{CS},\sqsubseteq)$ is an amalgamation class, in the sense of Theorem \ref{genFr}, where $\sqsubseteq$ denotes the subsemigroup relation restricted to semigroups from $\mathcal{CS}$. This is straightforward to prove by combining results from \cite{Imaoka} with \cite[Theorem 4]{Clarke}. The class $\mathcal{CS}$ also clearly satisfies (N1)--(N3) and from this, combined with Theorem \ref{genFr} the following result readily follows.  


\begin{thm}
Up to isomorphism, there is a unique  countable universal $\mathcal{CS}$-homo\-ge\-neous completely 0-simple 
semigroup $\mathcal{S}$.
\end{thm}

\begin{problem}
Is it true that every principal factor of $\cT$ is isomorphic to the semigroup $\mathcal{S}$?
\end{problem}

\subsection{The right and left ideal structure of $\cT$}

We have seen that the $\J$-classes of $\cT$ have order type $(\mathbb{Q},\leq)$. For $\cI$ we proved that the semilattice of idempotents is isomorphic to the countable universal homogeneous semilattice. Since $\cI$ is an inverse semigroup it follows that both $\cI / \R$ and $\cI / \L$ are also isomorphic to the countable universal homogeneous semilattice. It is reasonable to ask whether either or both of the posets $(\cT / \R, \leq_{\R})$ or  $(\cT / \L, \leq_{\L})$ admit a similar nice description in terms of homogeneous structures. First of all, since $\cT\cong\cT^{opp}$, we have a poset isomorphism $(\cT/\L,\leq_\L)\cong(\cT/\R,\leq_\R)$. One natural guess might have been that this is the countable generic partially ordered set, that is, the \Fr limit of the class of finite partially ordered sets. The same could be conjectured for the \emph{Rees order} of idempotents of $\cT$, $(E(\cT),\leq)$, defined by $e\leq f$ if and only if $ef=fe=e$.

These conjectures turn out not to be true. Namely, we have: 

\begin{pro}\label{prop_V}
Neither $(\cT/\L,\leq_\L)$ nor $(E(\cT),\leq)$ is isomorphic to the countable universal homogeneous poset.
\end{pro}

\begin{proof}
Since $\cT$ is universal we can find distinct $e,f,g\in E(\cT)$ with $ef=fe=g$ (i.e.\ a copy of $V_3$, the 
non-chain 3-element semilattice, in $\cT$). Then $ge=g=gf$, so $L_g<L_e$, $L_g<L_f$ and $L_e\parallel L_f$. 
Now suppose, for the sake of a contradiction, that $P=(S/\L,\leq_\L)$ is isomorphic to the countable universal 
homogeneous poset. It follows from \cite[Axiom ${}_1\phi_0^2$]{AB} that there is an element $h\in E(\cT)$ 
such that $L_h\leq L_e$, $L_h\leq L_f$ and $L_h\parallel L_g$. If we had such an $h$ then $h\leq_\L e$ would 
imply $he=h$ and similarly we would conclude $hf=h$. It would follow $hg=hef=hf=h$, that is $h\leq_\L g$, 
a contradiction.

This argument easily adapts to the Rees order of idempotents of $\cT$.
\end{proof}

We currently do not know a good characterisation of the poset $(\cT/\L,\leq_\L)$. 
It may be shown that this poset has the following properties: (i) it is universal, (ii) it is dense and without minimal or maximal elements, and (iii) $\Aut(\cT)$ acts transitively on finite chains of $\L$-classes.  It may also be shown that the Rees order of $E(\cT)$ has the properties analogous to (i)--(iii). 
\begin{problem}
Give descriptions of the posets $(\cT /\L,\leq_\L)$ and $(E(\cT),\leq)$. Are they isomorphic to restricted \Fr limits of some kind, in the sense of Theorem~\ref{genFr}? 
\end{problem}

\section{Subsemigroups of $\cT$ and $\cI$}
\label{sec_countembed}

In the main results above, Theorems~\ref{prop_i} and \ref{prop_t}, we saw that $\cI$ embeds every finite inverse semigroup, and $\cT$ embeds every finite semigroup. As mentioned in the introduction above, Hall's group $\cU$ not only embeds every finite group, but also embeds every countable locally finite group. Given this, it is natural to ask: which semigroups embed into $\cT$, and which inverse semigroups embed into $\cI$? These are more difficult questions than the corresponding question for $\cU$. We shall discuss the problem for $\cT$. The situation for $\cI$ is similar.    

Every subsemigroup of $\cT$ is of course countable and locally finite. We do not know whether the converse is true. 
%
%
\begin{problem}\label{prob_A}
Does every countable locally finite semigroup embed into $\cT$? 
\end{problem}
%
%
Let $S$ be a countably infinite locally finite semigroup. We can write $S = \bigcup_{i \geq 0} S_i$ where the $S_i$ are distinct finite subsemigroups of $S$ with $S_i \subseteq S_{i+1}$ for all $i$. Since each $S_i$ is finite we know that for each $i \geq 0$ there is an embedding $\phi_i: S_i \rightarrow \cT$. If each $S_i$ belonged to the class $\cB$ then using the extension property from Theorem~\ref{genFr} it is not hard to see that the embeddings $\phi_i$ could be chosen in such a way that for each $i$, the map $\phi_i:S_i \rightarrow \cT$ is the restriction of $\phi_{i+1}:S_{i+1} \rightarrow \cT$. In this way the union of the maps $\phi_i$ for $i \geq 0$ would define an embedding of $S$ into $\cT$. Let us say that a semigroup $S$ is a $\cB$-limit semigroup if $S$ is the direct limit of a countable chain of embeddings of distinct finite semigroups  
\[
S_0 \rightarrow S_1 \rightarrow S_2 \rightarrow \ldots 
\] 
where $S_i \in \cB$ for all $i \geq 0$. The above discussion shows that every $\cB$-limit semigroup embeds into $\cT$. Of course, more generally, any subsemigroup of a $\cB$-limit semigroup embeds into $\cT$. This leads to the following straightforward result. 

\begin{pro}\label{prop_OK}
Let $S$ be a countably infinite semigroup. The following are equivalent: 
\begin{enumerate}
\item $S$ embeds into $\cT$; 
\item $S$ embeds into some $\cT_n$-limit semigroup; 
\item $S$ embeds into some $\cB$-limit semigroup.
\end{enumerate}
\end{pro}
\begin{proof}
Since $\cT$ is a $\cT_n$-limit semigroup, and since $\cT_n \in \cB$ for all $n$, it follows that (1) implies (2) and that (2) implies (3). The fact that (3) implies (1) follows from the discussion preceding the statement of the proposition.  
\end{proof} 
By Cayley's Theorem every finite semigroup embeds into some finite full transformation semigroup. From above we are led to the following related question: 

\begin{problem}\label{prob_B}
Is every countable locally finite semigroup isomorphic to a subsemigroup of some full transformation limit semigroup? 
\end{problem} 
%
%
%
%


%
%
Related to the question of which semigroups embed into $\cT$, more generally we do not know the answer to the following very natural question:  

\begin{problem}
Does there exist a countable semigroup which embeds every countable locally finite semigroup? 
\end{problem} 

For the inverse semigroup $\cI$ the obvious analogue of Proposition~\ref{prop_OK} holds, and we have the corresponding open questions.  

\begin{problem}
Does every countable locally finite inverse semigroup embed into $\cI$? 
%
%
\end{problem} 


\section{The semigroups $\cC_n$ and $\cV_n$} \label{sec_cnvn}

As explained in the introduction, P. Hall's universal group in not only an $\cS_n$-limit group, but it can be expressed as an $\cS_n$-limit in a particularly nice way via successive embeddings arising from repeated application of Cayley's Theorem. It is natural to look for correspondingly nice descriptions of  $\cT$ as a $\cT_n$-limit semigroup, and of $\cI$ as a $\cI_n$-limit inverse semigroup. We currently do not know of any such nice descriptions. The purpose of this section is to explain why the obvious approach of trying to obtain $\cT$ by repeated iteration of Cayley's Theorem for semigroups does not work, and the corresponding observation for $\cI$.   

Fix a natural number $n$. Let $S_0 = \cT_n$ and then define the directed chain of full transformation semigroups
\begin{align*}
& S_0 \to S_1 \to S_2 \to \dots 
\end{align*}
where for each $i\geq 0$ we have $S_{i+1}=\cT_{S_i}$ and the embeddings are given in each case by taking the right regular representation. We use $\cC_n$ to denote the direct limit of this chain of semigroups. It is easy to see that for all natural numbers $n$ the semigroup $\cC_n$ is a monoid.
On the other hand, the semigroup $\cT$ is not a monoid. Indeed, since $\cT$ is universal it contains at least two distinct idempotents. Also, since $\Aut(\cT)$ acts transitively on members of $\cB$ in $\cT$ it follows that $\Aut(\cT)$ acts transitively on the set of idempotents $E(\cT)$ of $\cT$. This would be impossible if $\cT$ were a monoid, since any automorphism of a monoid must clearly fix its identity element. 
Therefore $\cT$ cannot be isomorphic to $\cC_n$ for any value of $n$. However, there is a natural monoid analogue $\cT'$ of the semigroup $\cT$ with the same definition but working in the category of monoids, and working with submonoids.
%
%
One could then ask whether or not $\cT'$ and $\cC_n$ are isomorphic for some value of $n$. 
In a similar way as for the semigroup $\cT$, one may show that every maximal subgroup of $\cT'$ is isomorphic to Hall's universal homogeneous group. We shall now show that this is not true of the semigroup $\cC_n$, and thus $\cT'$ and $\cC_n$ cannot be isomorphic.  

To this end, let $S=\cT_n$ and let $\psi:S\to\cT_S$ be the right regular representation map where $\alpha \psi = \rho_\alpha$. For any $\alpha \in \cT_n$ set 
\[
\fix\alpha = \{ i \in [n] : i \alpha = i \} 
 \quad \mbox{and} \quad 
\fix\rho_\alpha = \{ \gamma \in \cT_n : \gamma \alpha = \gamma \}.  
\]  
For any $\alpha \in \cT_n$ we have 
\[
\fix\rho_\alpha = \{ \gamma \in \cT_n : \im \gamma \subseteq \fix \alpha \}. 
\]
This is because 
$\gamma \alpha = \gamma$
holds if and only if 
$\alpha$ acts as the identity on $\im \gamma$, that is,
$\im \gamma \subseteq \fix \alpha$. 
The size of the set $\fix\rho_\alpha$ is of course just the number of maps from an $n$-element set into a set of size $|\fix\alpha|$. From this it immediately follows that if $|\fix\alpha| < |\fix\beta|$ then 
$|\fix\rho_\alpha| < |\fix\rho_\beta|$. 
Now 
%
%
suppose $1<r<n$ and $r\geq 5$. Fix an idempotent
\[
e=
\left(\begin{array}{cccccccc}
1 & 2 & 3 & \cdots & r & r+1 & \cdots 	& n\\
1 & 2 & 3 & \cdots & r & r & \cdots 		& r
\end{array}\right).
\]
Fix two elements of order two in the maximal subgroup $H_e$ defined as follows: 
%
%
\[
\alpha =  
\begin{pmatrix}
1 & 2 & 3 & 4 & \cdots & r & r+1 & \cdots 	& n\\
2 & 1 & 3 & 4 & \cdots & r & r & \cdots 	& r
\end{pmatrix}, 
\; \beta = 
\begin{pmatrix}
1 & 2 & 3 & 4 & 5 &\cdots & r & r+1 & \cdots 	& n\\
2 & 1 & 4 & 3 & 5 & \cdots & r & r & \cdots 	& r
\end{pmatrix}.
\]
We use the shorthand $\alpha=(12)$ and $\beta=(12)(34)$ for these elements of the maximal subgroup $H_e$. The elements $\alpha$ and $\beta$ both have order $2$, but they are not conjugate in $H_e$ since they have a different number of fixed points. By the observations above, $\alpha \psi$ and $\beta \psi$ are two elements in the $\H$-class $H_{e \psi}$ of $\cT_{\cT_n}$ both of order two, but they are not conjugate in $H_{e \psi}$ since they have different numbers of fixed points. Repeating this argument, we conclude that for every $i \geq 1$ the images of the elements $\alpha$ and $\beta$ under the embeddings
$
S_0 \rightarrow \ldots \rightarrow S_i
$
are not conjugate inside the unique maximal subgroup of $S_i$ to which they both belong. It follows that the elements that $\alpha$ and $\beta$ represent in $\cC_n$ both have order two, but are not conjugate in the maximal subgroup of $\cC_n$ to which they belong. But in Hall's group any two elements of order two are conjugate to each other. It follows that the maximal subgroup of $\cC_n$ containing $e \in S_0 \subseteq \cC_n$ is not isomorphic to Hall's group. This argument can easily be extended to prove the same result starting with any $e \in S_0 = \cT_n$ of rank $r$, where $2 \leq r \leq n-1$. Thus for all $n \geq 2$ the semigroup $\cC_n$ is not isomorphic to the universal maximally homogeneous locally finite monoid $\cT'$. 
%
%
An easy adaptation of the above argument can be used to prove that the 
inverse semigroup $\cV_n$ 
%
is never isomorphic to the universal locally finite maximally homogeneous inverse monoid $\cI'$

We do not know much about the semigroups $\cC_n$ and $\cV_n$. 
%
%
%
\begin{problem}
Do we have $\cC_n \cong \cC_m$ for all $n, m \geq 1$? We ask the same question for $\cV_n$. 
\end{problem}
It is not too hard to prove 
that for any $n\geq 2$, both $(\cC_n/\J,\leq)$ and $(\cV_n/\J,\leq)$ are isomorphic to $1+\mathbb{Q}+1$.  
Also, the ideas in the proof of Theorem~\ref{prop_i}(4) suffice to prove that the semilattice $(E(\cV_n),\leq)$
satisfies condition ($\ast$) in part (iii) of Theorem~\ref{Omega}. However, $E(\cV_n)$ does have unique maximal and minimal idempotents.  
By applying results and arguments from \cite{Droste},  \cite{AB} and \cite[pages 33-34]{DKT} it may be shown that $E(\cV_n)$ is isomorphic to any nontrivial interval in
$\Omega$, which in turn is isomorphic to the countable universal homogeneous semilattice monoid with zero.
%
%
%
%
%


\begin{acknowledgements}
This work was supported by the London Mathematical Society Research in Pairs (Scheme 4) grant ``Universal locally finite partially homogeneous
semigroups and inverse semigroups'' (Ref:\ 41530), to fund a 9-day research visit of the first named author to the University of East Anglia 
(Summer 2016). The research of I.\ Dolinka was supported by the Ministry of Education, Science, and Technological Development of the 
Republic of Serbia through the grant No.174019. This research of R.\ D. Gray was supported by the EPSRC grant EP/N033353/1 ``Special inverse monoids: subgroups, structure, geometry, rewriting systems and the word problem''.
The authors would like to thank the anonymous referee for their helpful comments, and also to thank the handling editor of the paper Manfred Droste whose comments led to the questions that are now posed in Section~8.  
\end{acknowledgements}


\end{document}